\documentclass[12pt,reqno]{amsart}

\usepackage{amsfonts,amssymb,amsthm}
\usepackage{mathdots}
\usepackage[all]{xy}
\usepackage{blkarray,longtable}
\usepackage{mathrsfs,calligra}
\usepackage[bbgreekl]{mathbbol}
  \DeclareSymbolFontAlphabet{\mathbb}{AMSb}
  \DeclareSymbolFontAlphabet{\mathbbl}{bbold}
\usepackage{xcolor}
\usepackage{graphicx} 
\usepackage{here,float}

\usepackage{enumitem}

\usepackage{pgf}

\usepackage[vcentermath,enableskew]{youngtab}
\usepackage{ytableau}
\ytableausetup{centertableaux,nobaseline}

\usepackage[
backend=biber,
style=ieee-alphabetic,
sorting=nyt,
url=false,
isbn=false,
eprint=false,
doi=false
]{biblatex}
\makeatletter
\def\input@path{{./}}
\graphicspath{{./figures/}}
\makeatother

\newcommand{\version}{Ver.~0.0}
\newcommand{\setversion}[1]{\renewcommand{\version}{Ver.~{#1}}}
\setversion{0.0 [2024/05/21 18:42:41 JST]}
\setversion{0.7 [2025/03/11 16:28:42 JST]}
\setversion{0.8 [2025/08/04 17:48:44 JST]}
\setversion{0.9 [2025/08/13 16:37:15 JST]}
\setversion{0.99 [2026/09/04 10:14:59 JST]}
\setversion{0.999 [2026/09/09 07:34:06 JST]}

\title{Double flag varieties of Levi type with a finite number of orbits}
\date{\today}

\author{Lucas Fresse and Kyo Nishiyama}
\address{Universit\'e de Lorraine, CNRS, Institut \'Elie Cartan de Lorraine, UMR 7502, Vandoeu\-vre-l\`es-Nancy, F-54506, France}
\email{lucas.fresse@univ-lorraine.fr}
\address{Department of Mathematics, Aoyama Gakuin University, Fuchinobe 5-10-1, Chuo-ku, Saga\-mihara 252-5258, Japan}
\email{kyo.nishiyama@gmail.com}
\thanks{K.~N.~is supported by JSPS KAKENHI Grant Number \#{25K06938}.}

\numberwithin{equation}{section}

\newtheorem{theorem}{Theorem}[section]
\newtheorem{lemma}[theorem]{Lemma}
\newtheorem{proposition}[theorem]{Proposition}
\newtheorem{corollary}[theorem]{Corollary}

\theoremstyle{definition}
\newtheorem{example}[theorem]{Example}

\newtheorem{definition}[theorem]{Definition}
\newtheorem{remark}[theorem]{Remark}

\newtheorem{problem}[theorem]{Problem}

\newcounter{penum}
\newenvironment{penumerate}{%
\par\smallskip
\begin{list}{$\;\;(\thepenum)$}{%
\usecounter{penum}
\setlength{\topsep}{0pt}
\setlength{\partopsep}{0pt}
\setlength{\parsep}{1ex}
\setlength{\itemindent}{0pt}
\setlength{\labelsep}{.5em}
\setlength{\labelwidth}{0pt}
\setlength{\leftmargin}{0pt}
\setlength{\rightmargin}{0pt}
\setlength{\itemsep}{0pt}
}}
{\end{list}\par\medskip}

\newcommand{\skipover}[1]{}

\newcommand{\mbfa}{\mathbf{a}}
\newcommand{\mbfb}{\mathbf{b}}
\newcommand{\mbfc}{\mathbf{c}}
\newcommand{\mbfd}{\mathbf{d}}
\newcommand{\mbfe}{\mathbf{e}}
\newcommand{\mbff}{\mathbf{f}}
\newcommand{\mbfg}{\mathbf{g}}
\newcommand{\mbfh}{\mathbf{h}}
\newcommand{\mbfm}{\mathbf{m}}

\newcommand{\Xfv}{\mathfrak{X}}
\newcommand{\Grass}{\mathrm{Gr}}

\newcommand{\GL}{\mathrm{GL}}

\newcommand{\K}{\mathbb{K}}

\newcommand{\Hom}{\qopname\relax o{Hom}}

\newcommand{\Stab}{\qopname\relax o{Stab}}

\newcommand{\Rep}{\qopname\relax o{Rep}}

\newcommand{\Flags}{\mathscr{F}\!\ell}

\newcommand{\quiver}{\mathscr{Q}}
\newcommand{\bfell}{\boldsymbol{\ell}}

\newcommand{\nopicture}[1]{}

\begin{document}

\begin{abstract}
We consider a double flag variety of the form $L/Q\times G/P$ where $G$ is a connected reductive group and $L$ is a Levi subgroup of $G$.
The Levi subgroup $L$ acts diagonally on this double flag variety and a basic problem is to classify pairs of parabolic subgroups $P\subset G$ and $Q\subset L$ such that there are finitely many orbits for the considered action. We solve this problem in the case where $G$ is the general linear group by using representations of quivers.
\end{abstract}

\maketitle

\section{Introduction}

Let $G$ be a connected reductive algebraic group over $\K$, which (throughout this paper) denotes an algebraically closed field of characteristic zero.
When $P\subset G$ is a parabolic subgroup, the flag variety $G/P$ is a complete homogeneous space, and every complete homogeneous space for $G$ is of this form.

When $P_1,P_2\subset G$ is a pair of parabolic subgroups, the double flag variety $G/P_1\times G/P_2$ is equipped with the diagonal action of $G$. It follows from the Bruhat decomposition that there are finitely many orbits for this action, which are parametrized by double cosets in the Weyl group.

Multiple flag varieties of the form $\prod_{i=1}^\ell G/P_i$ with $\ell\geq 3$ have infinitely many orbits in general.
In the classical cases, Magyar--Weyman--Zelevinsky \cite{MWZ.1999,MWZ.2000} and
Matsuki \cite{Matsuki.2015,Matsuki.arXiv2019} classified multiple flag varieties of finite type, i.e., with a finite number of $G$-orbits. In particular, if $\ell>3$ (and the $P_i$'s are proper parabolic subgroups), there are always infinitely many orbits.

In the present paper, we consider another type of double flag variety, namely
\begin{equation*}
\Xfv=L/Q \times G/P
\end{equation*}
where $L\subset G$ is a Levi subgroup, and $P$ and $Q$ are parabolic subgroups of $G$ and $L$, respectively.
The variety $\Xfv$ is endowed with the diagonal action of $L$,
and our problem is the following.

\begin{problem}\label{problem:classify.MFV.of.finite-type}
For a given $G$,
classify all triples $(L,P,Q)$ such that $\Xfv$ has a finite number of $L$-orbits.
\end{problem}

Note also that there is a natural bijection between $L$-orbits of $\Xfv=L/Q\times G/P$ and $Q$-orbits of $G/P$.
Thus, in Problem \ref{problem:classify.MFV.of.finite-type}, $\Xfv$ will have finitely many $L$-orbits if and only if $G/P$ has finitely many $Q$-orbits.

In the rest of the paper, we focus on the case of type A, where
\begin{equation}
G = \GL_n(\K) , \qquad
L = \prod_{i=1}^m \GL_{p_i}(\K) \quad \text{with} \quad \sum_{i=1}^m p_i = n ,
\end{equation}
and $ L $ is diagonally embedded into $ G $.
We give a complete solution to Problem \ref{problem:classify.MFV.of.finite-type} for $G=\GL_n(\K)$,
and the answer appears in Theorem \ref{T:main_theorem} below.
But before stating the theorem, let us summarize results which are already known on Problem \ref{problem:classify.MFV.of.finite-type}.

For $G=\GL_n(\K)$ (in short, $\GL_n$), every parabolic subgroup of $G$ is up to conjugation the subgroup $P=P_{\mbfd}$ of invertible blockwise upper triangular matrices with diagonal blocks of sizes $d_1,\ldots,d_k$, where $\mbfd=(d_1,\ldots,d_k)$ is some composition of $n$.
Then the flag variety $G/P_{\mbfd}$ can be identified with the variety $\Flags(\mbfd)$ of flags $(F_0\subset F_1\subset \ldots\subset F_k=\K^n)$ with $\dim F_i/F_{i-1}=d_i$ for all $i$. Special cases of flag varieties are the following.
\begin{itemize}
    \item If $\mbfd=(1^n)$ (where $1^n$ means $1$ repeated $n$ times), the variety $G/P=\Flags(1^n)$ consists of full flags
    $(F_i)_{i=0}^n$, with $\dim F_i=i$ for all $i$. In this case, $P=P_{(1^n)}$ is the Borel subgroup $B\subset \GL_n$ of invertible upper triangular matrices.
    \item At the other extreme, for $\mbfd=(d,n-d)$, the variety $G/P$ coincides with the Grassmannian variety $\Grass(d;n)$ of $d$-dimensional subspaces in $\K^n$.
    Correspondingly $P=P_{(d,n-d)}$ has two diagonal blocks and is called a maximal parabolic subgroup.
    \item An even more extreme case is when $\mbfd=(1,n-1)$ or $(n-1,1)$, so that $G/P$ is isomorphic to the projective space $\mathbb{P}^{n-1}$;
    in this case $P=P_{\mbfd}$ is the stabilizer of a line or of a hyperplane, and it is called mirabolic.
\end{itemize}
A parabolic subgroup of $L=\prod \GL_{p_i}$ takes the form
$Q=Q_1\times\cdots\times Q_m$ where $Q_i$ is a parabolic subgroup of $\GL_{p_i}$ for all $i$. Thus $Q_i$ is associated with some composition $\mbfa_i$ of $p_i$, and we have
$$
\Xfv=L/Q\times G/P=\Flags(\mbfa_1)\times\cdots\times\Flags(\mbfa_m)\times\Flags(\mbfd).
$$

If $P=G$ then $\Xfv=L/Q$ consists of one $L$-orbit. In the considerations below, we assume $P\subsetneq G$ and $p_i\geq 1$ for all $i$, thus $m$ is the number of diagonal blocks of $L$.
\begin{itemize}
    \item[\rm (a)]
    When $P$ is mirabolic, that is, when $G/P$ is a projective space, then $G/P$ has finitely many $T$-orbits for any maximal torus $T\subset G$, thus it has a fortiori finitely many $Q$-orbits for any parabolic subgroup $Q\subset L$ because $Q$ contains a maximal torus of $G$. Therefore, $\Xfv$ has always finitely many $L$-orbits in this case, whatever are $L$ and $Q$.
    \item[\rm (b)]
    Duckworth \cite{Duckworth} classified the pairs $(L,P)$, formed by a Levi subgroup and a parabolic subgroup of $G$, such that $L$ has finitely many orbits on $G/P$. His result concerns all classical groups but we focus here on the case of $G=\GL_n$.
In this case he obtained that $G/P$ has finitely many $L$-orbits if and only if $m\leq 2$ or ($m=3$ and $P$ is maximal) or ($m\geq 4$ and $P$ is mirabolic). This implies that $\Xfv$ will a fortiori have infinitely many $L$-orbits if $m\geq 4$ and $P$ is not mirabolic, or $m\geq 3$ and $P$ is not maximal.
    \item[\rm (c)] When $L$ has only $m=2$ blocks, $L$ is a symmetric subgroup of $G$ and the pair $(G,L)$ is called a symmetric pair of type AIII.
This case is special but one of most interesting cases.  A complete classification of the pair $ (P, Q) $, which gives solutions to
Problem \ref{problem:classify.MFV.of.finite-type} is obtained in \cite[Corollary 8.6]{Fresse.Nishiyama.Overview.2023}, based on results of \cite{Homma.2021}; see also Table \ref{table:finite.type.m=2} below.
    \item[\rm (d)] Finally, if $L$ has only one block, i.e., $L=G$, then $\Xfv=G/Q\times G/P$ becomes a double flag variety for $G$, and it has finitely many orbits due to the Bruhat decomposition.
\end{itemize}
Based on these observations,
in order to answer Problem \ref{problem:classify.MFV.of.finite-type} for $G=\GL_n$, it remains to consider the case where $L$ has $m=3$ blocks and $G/P=\Grass(d;n)$ with $2\leq d\leq n-2$.
Our main task in the present paper is to settle this case, which is done in Section \ref{section:classification.proofs}; specifically, our classification result regarding this case is Theorem \ref{Thm:classification.finite.type.MFV} therein, and we reproduce it in part (C) of Theorem \ref{T:main_theorem} below. By compiling Theorem \ref{Thm:classification.finite.type.MFV} with (a)--(d) above, we indeed obtain the next theorem which answers Problem \ref{problem:classify.MFV.of.finite-type}.

\begin{theorem}\label{T:main_theorem}
Let $G=\GL_n$, a Levi subgroup $L=\prod_{i=1}^m \GL_{p_i}$ with $p_i\geq 1$, $\sum_{i=1}^m p_i=n$, and parabolic subgroups
$P\subsetneq G$ and $Q=\prod_{i=1}^m Q_i\subset L$ with $Q_i\subset \GL_{p_i}$.
The double flag variety $\Xfv=L/Q\times G/P$ has a finite number of $L$-orbits if and only if one of the following conditions is satisfied.
\begin{itemize}
\item[\rm (A)] $m=1$, that is, $L=G$.
\item[\rm (B)] $m=2$ and the triple $(P,Q_1,Q_2)$ appears (up to switching the roles of $Q_1$ and $Q_2$) in Table \ref{table:finite.type.m=2}.
\item[\rm (C)] $m=3$, $P$ is maximal (i.e., $G/P$ is a Grassmannian variety), and the quadruple $(P,Q_1,Q_2,Q_3)$ appears (up to switching the roles of $Q_1,Q_2,Q_3$) in Table \ref{table:finite.type.m=3}.
\item[\rm (D)] $m\geq 4$ and $P$ is mirabolic, i.e., $G/P$ is a projective space.
\end{itemize}
\end{theorem}

\begin{table}[htbp]
\caption{Double flag varieties of finite type for $m=2$}\label{table:finite.type.m=2}
%
\begin{tabular}{||c|c|c||}
\hline
$P$ & $Q_1$  &  $Q_2$ \\ \hline \hline
maximal & any & any \\ \hline
3 blocks & maximal & $\leq$ 4 blocks \\ \hline
3 blocks & $\GL_{p_1}$ or mirabolic & any \\ \hline
3 blocks, one of size $2$ & $\GL_{p_1}$ or maximal  & any \\ \hline
3 blocks, one of size $1$ & any & any \\ \hline
4 blocks & $\GL_2$ & any \\ \hline
4 blocks & $\GL_{p_1}$ & $\leq$ 4 blocks \\ \hline
4 blocks, one of size $1$ & $\GL_{p_1}$ & any \\ \hline
5 or 6 blocks & $\GL_{p_1}$ & maximal \\ \hline
any & $\GL_1$ & any \\ \hline
any & $\GL_2$ & maximal \\ \hline
any & $\GL_{p_1}$ & $\GL_{p_2}$ or mirabolic \\ \hline
\end{tabular}
\hfil
\end{table}

\begin{table}[htbp]
\caption{Double flag varieties of finite type for $m=3$}\label{table:finite.type.m=3}
%
\begin{tabular}{||c|c|c|c||}
\hline
$P=P_{(d_1,d_2)}$ (maximal) & $Q_1$  &  $Q_2$ & $Q_3$ \\ \hline \hline
$\min\{d_1,d_2\}\leq 2$ & any & any & any \\ \hline
$\min\{d_1,d_2\}=3$ & $\GL_{p_1}$ & any & any \\ \hline
$\min\{d_1,d_2\}=4$ or $5$ & $\GL_{p_1}$ & maximal & any \\ \hline
any & $\GL_1$ & any & any \\ \hline
any & $\GL_2$ & maximal & any \\ \hline
any & $\GL_{p_1}$ & $\GL_{p_2}$ or mirabolic & any \\ \hline
any & $\GL_{p_1}$ & maximal & $\leq$ 4 blocks.
 \\ \hline
\end{tabular}
\hfil
\end{table}

\begin{remark}
The tables show some cases where $Q$ can be a Borel subgroup of $L$:
for $P$ maximal or $P$ has three blocks with (one block of size 1 or $p_1=2$) or $Q_1=\GL_1$ in Table \ref{table:finite.type.m=2};
for $\min\{d_1,d_2\}\leq 2$ or $Q_1=\GL_1$ in Table \ref{table:finite.type.m=3}. Then the variety $G/P$ is $L$-spherical. These cases appear  in the classification given in \cite{Avdeev.Petukhov.2014} of pairs $(H,G/P)$ where $H\subset G$ is a reductive subgroup and $G/P$ is a $H$-spherical partial flag variety.
\end{remark}

\subsection*{Outline of the arguments}
As explained above, for completing the proof of Theorem \ref{T:main_theorem}, it remains to consider the case where $L$ has $m=3$ blocks and $G/P=\Grass(d;n)$ with $2\leq d\leq n-2$; hence
$$
L=\GL_p\times \GL_q\times \GL_r,\quad p,q,r\geq 1,\quad p+q+r=n,
$$
and the multiple flag variety under consideration is
\begin{equation}
\label{X:3blocks}
\Xfv=\Flags(\mbfa)\times \Flags(\mbfb)\times \Flags(\mbfc)\times \Grass(d;n) ,
\end{equation}
where $\mbfa,\mbfb,\mbfc$ are compositions of $p,q,r$, respectively.

Our main task in this paper is therefore to classify the quadruples $(\mbfa,\mbfb,\mbfc,d)$ for which $\Xfv$ of (\ref{X:3blocks}) has a finite number of $L$-orbits.
For doing this, inspired by the strategy of Magyar--Weyman--Zelevinsky \cite{MWZ.1999}
(and also Homma \cite{Homma.2021}; see part (c) above),
we use the theory of representations of quivers.
Key ingredient of this approach is a theorem due to Kac \cite{Kac.Quiver.1980} 
on the relation between indecomposable objects in the category of representations of a quiver $\quiver$ and real and imaginary roots of some corresponding Kac--Moody Lie algebra.
As in \cite{MWZ.1999} and \cite{Homma.2021}, our problem is reduced to the problem of classifying isomorphism classes in a suitable subcategory of representations of quivers with injective arrows.

In Section \ref{section:2}, we give a short review of the theory of representations of quivers (in the case of a star-shaped quiver with four branches)
including 
a general statement which establishes
a bijection between orbits and
decompositions of objects in the quiver category.
Main results are summarized into Theorem \ref{theorem:finiteness.criterion}.

In Section \ref{section:classification.proofs},
we obtain a complete classification of multiple flag varieties of the form (\ref{X:3blocks}) which have finitely many orbits
(Theorem \ref{Thm:classification.finite.type.MFV}).

Towards the description of orbits, in Section \ref{section:4},
we give a complete list of (parameters of) indecomposable objects in the category of representations of quivers which can arise when $\Xfv$ has a finite number of orbits (Theorem \ref{T:indecomposables}).
We derive a parametrization of the orbits in Corollary \ref{C:orbits}.

In Section \ref{section:5}, we illustrate these results by giving a combinatorial parametrization of the orbits of
$L=\GL_p\times \GL_q\times \GL_r$ (with $p+q+r=n$)
on the multiple flag variety
$$
\Flags(1^p)\times \Flags(1^q)\times \Flags(1^r)\times \Grass(2;n)
$$
(this case corresponds to the first line of Table \ref{table:finite.type.m=3}).

\section{The representations of star-shaped quivers}

\label{section:2}

Fix a quadruple of positive integers
$ \bfell=(\ell_1,\ell_2,\ell_3,\ell_4) $.
We consider a star-shaped quiver $ \quiver=\quiver(\bfell) $ with four branches of lengths $\ell_1,\ell_2,\ell_3,\ell_4$,
each branch consisting of $\ell_i-1$ arrows leading from an endpoint to the center.
See the figure below.
\begin{equation}
\label{eq:quiver}
\xymatrix{
\stackrel{A_1}{\bullet} \ar^{\alpha_1}[r] & \stackrel{A_2}{\bullet} \ar[r] & \cdots \ar[r] & \stackrel{A_{\ell_1-1}}{\bullet} \ar^{\alpha_{\ell_1-1}}[rd] & \\
\stackrel{B_1}{\bullet} \ar^{\beta_1}[r] & \stackrel{B_2}{\bullet} \ar[r] & \cdots \ar[r] & \stackrel{B_{\ell_2-1}}{\bullet} \ar^{\beta_{\ell_2-1}}[r] & \stackrel{D_{\ell_4}}{\bullet} & \stackrel{D_{\ell_4-1}}{\bullet} \ar[l] & \cdots \ar[l] & \stackrel{D_2}{\bullet} \ar[l] & \stackrel{D_1}{\bullet} \ar^{\delta_1}[l] \\
\stackrel{C_1}{\bullet} \ar^{\gamma_1}[r] & \stackrel{C_2}{\bullet} \ar[r] & \cdots \ar[r] & \stackrel{C_{\ell_3-1}}{\bullet} \ar_{\gamma_{\ell_3-1}}[ru] & \\
}
\end{equation}
We name the vertices of the quiver by using labels $ A_1, A_2, \dots $, and the arrows are named $ \alpha_1, \alpha_2, \dots $.
Sometimes we call the central vertex $ D_{\ell_4} $ as $ A_{\ell_1} $, and similarly as
$ B_{\ell_2} $ or $ C_{\ell_3} $.  They all denote the same vertex.
We will write $X_i$ for referring to a general vertex of $\quiver$, and $\chi_i$ for a general arrow (i.e., $X$ stands for one of the letters $A$, $B$, $C$, or $D$, and $\chi$ for $\alpha$, $\beta$, $\gamma$, or $\delta$).

A (finite-dimensional) representation of the quiver $\quiver$ is a functor $\pi$ from $\quiver$ to the category of (finite-dimensional) vector spaces;
in other words it is a collection of vector spaces $\pi(X_i)$ labeled by the vertices of $\quiver$ and linear maps $\pi(\chi_i)$ labeled by the arrows
of $\quiver$. A morphism of representations $\phi:\pi\to\pi'$ is a morphism of functors, i.e., a collection of linear maps $\phi(X_i):\pi(X_i)\to\pi'(X_i)$ which commute with the $\pi(\chi_j)$'s and the $\pi'(\chi_k)$'s. If the linear maps $\phi(X_i)$ are bijective, we say that $\phi$ is an isomorphism and in this case it has an inverse morphism $\phi^{-1}:\pi'\to\pi$.

The representations of $\quiver$
form an abelian category which we denote by $\Rep\quiver$.
Furthermore, we consider the additive full subcategory $\Rep'\quiver$ formed by representations $\pi$ such that
\begin{equation}
\label{repQ'}
\begin{array}{ll}
 & \mbox{$\pi(\chi_i)$ is injective for every arrow $\chi_i$ in $\quiver$}
\\[2mm]
\text{and}
&
\pi(D_{\ell_4})=\mathrm{Im}\,\pi(\alpha_{\ell_1-1})\oplus \mathrm{Im}\,\pi(\beta_{\ell_2-1})
\oplus \mathrm{Im}\,\pi(\gamma_{\ell_3-1}).
\end{array}
\end{equation}
Note that $ \Rep'\quiver$ is not an abelian category.
Given a representation $\pi\in\Rep \quiver$,
by a slight abuse of terminology
we call dimension vector of $\pi$
the sequence
\begin{eqnarray*}
\lambda(\pi)
 & = &
\left(
\left(\dim\pi(A_i)-\dim\pi(A_{i-1})\right)_{i=1}^{\ell_1},
\left(\dim\pi(B_i)-\dim\pi(B_{i-1})\right)_{i=1}^{\ell_2},
\right.
\\
 & &
\left.
\left(\dim\pi(C_i)-\dim\pi(C_{i-1})\right)_{i=1}^{\ell_3},
\left(\dim\pi(D_i)-\dim\pi(D_{i-1})\right)_{i=1}^{\ell_4}
\right)
\end{eqnarray*}
with convention $\pi(X_0)=0$ if $X\in\{A,B,C,D\}$.
(Thus the coefficients of $\lambda(\pi)$ are not the actual dimensions of the spaces $\pi(X_i)$,
but we can recover $\dim \pi(A_i)=a_1+\dots+a_i$ from
$\lambda(\pi)=\left((a_i)_{i=1}^{\ell_1}, (b_i)_{i=1}^{\ell_2},(c_i)_{i=1}^{\ell_3},(d_i)_{i=1}^{\ell_4}\right)$ for example.)
The dimension vectors of the objects $\pi\in\Rep'\quiver$ have nonnegative coefficients and they
form a semigroup which we describe now.

We define $ \Lambda=\Lambda(\bfell) $
as the set of nonzero quadruples
$\lambda=(\mbfa,\mbfb,\mbfc,\mbfd)$
of the following type:
\begin{enumerate}[label=($\Lambda$.\arabic*)]
\item \label{Lambda.property:item:1}
$\mbfa = (a_i)_{i=1}^{\ell_1}$, $\mbfb = (b_i)_{i=1}^{\ell_2}$,
$\mbfc = (c_i)_{i=1}^{\ell_3}$,
$\mbfd = (d_i)_{i=1}^{\ell_4}$ are
sequences of nonnegative integers.

\item \label{Lambda.property:item:2}
$|\mbfa|=|\mbfb|=|\mbfc|=|\mbfd|=:|\lambda|$, where $|\mbfa|=\sum_{i=1}^{\ell_1} a_i$
and $|\mbfb|$, $|\mbfc|$, $|\mbfd|$ are defined similarly.

\item \label{Lambda.property:item:3}
Write $\mbfa=(\mbfa',a'')$
with $\mbfa'=(a_1,\ldots,a_{\ell_1-1})$ and $a''=a_{\ell_1}$, and let us use similar notation for $\mbfb$ and $\mbfc$,
then
$|\mbfa'|+|\mbfb'|+|\mbfc'|=|\lambda|$;
equivalently this means that
$a''=|\mbfb'|+|\mbfc'|$, and similarly for $b''$ and $c''$.
\end{enumerate}
Thus $\Lambda$ is a semigroup for pointwise addition of quadruples
($\Lambda\cup\{0\}$ is a submonoid of the group $(\mathbb{Z}^{\ell_1+\ell_2+\ell_3+\ell_4},+)$).
%
%
We say that $\mu\in\Lambda$ is a \emph{summand} of $\lambda$ in $\Lambda$ if 
$\lambda-\mu\in\Lambda\cup\{0\}$, i.e., if we can write
$\lambda=\mu+\nu$ for some $\nu\in \Lambda\cup\{0\}$
(this is equivalent to the condition that
the tuple of integers $\mu$ is pointwise less or equal than $\lambda$, and in particular $\lambda$ is its own summand).
%
We furthermore define a quadratic form
$$
\Lambda=\Lambda(\bfell)\to\K,\quad \lambda=(\mbfa,\mbfb,\mbfc,\mbfd)\mapsto(\lambda|\lambda)
$$
called the \emph{Tits form},
which is given by 
\begin{equation} \label{eq:general.Tits.form}
(\lambda|\lambda) = \frac{1}{2}\Big(\sum_{i = 1}^{\ell_1} a_i^2+
\sum_{j = 1}^{\ell_2} b_j^2+\sum_{k = 1}^{\ell_3} c_k^2+\sum_{m = 1}^{\ell_4} d_m^2-2n^2\Big)
\end{equation}
where $n:=|\lambda|$.

Every $\lambda=(\mbfa,\mbfb,\mbfc,\mbfd)$ in $\Lambda(\bfell)$ determines a product of flag varieties
\begin{eqnarray}\label{eq:def.of.Xabcd}
\Xfv(\lambda) = \Xfv(\mbfa,\mbfb,\mbfc,\mbfd)
 & = & \GL_{|\mbfa'|}/Q_{\mbfa'} \times \GL_{|\mbfb'|}/Q_{\mbfb'} \times \GL_{|\mbfc'|}/Q_{\mbfc'} \times \GL_n/P_{\mbfd} \nonumber \\
 & = & \Flags(\mbfa')\times \Flags(\mbfb')\times
 \Flags(\mbfc')\times \Flags(\mbfd)
\end{eqnarray}
where $n=|\mbfa'| + |\mbfb'| + |\mbfc'|$, which is also equal to
$ |\mbfa|=|\mbfb|=|\mbfc|=|\mbfd| $ by properties
\ref{Lambda.property:item:2} and
\ref{Lambda.property:item:3}
of $ \lambda \in \Lambda(\bfell) $
(here $Q_{\mbfa'}$, $Q_{\mbfb'}$, $Q_{\mbfc'}$, $P_{\mbfd}$ are the standard parabolic subgroups
determined by the compositions $\mbfa',\mbfb',\mbfc',\mbfd$).
%
The variety $\Xfv(\lambda)$
is a slight generalization of the double flag variety $\Xfv$ of (\ref{X:3blocks}), where the Grassmannian variety $\Grass(d;n)$ is replaced (for the moment) by the general flag variety $\GL_n/P_\mbfd=\Flags(\mbfd)$.
As in (\ref{X:3blocks}),
we consider
the orbits of
$ L = \GL_{|\mbfa'|} \times \GL_{|\mbfb'|} \times \GL_{|\mbfc'|} $ on $\Xfv(\lambda)$.

Fix a decomposition
$$
V=V_1\oplus V_2\oplus V_3
$$
where
$ \dim V=n $ and
$ \dim V_1 = |\mbfa'| $, $ \dim V_2 = |\mbfb'| $,
$ \dim V_3 = |\mbfc'| $.
A point in $ \Xfv(\lambda) $
is a quadruple of flags
$ F = ( F^{(1)}, F^{(2)}, F^{(3)}, F^{(4)} ) $
where
$ F^{(j)} = ( F^{(j)}_i )_{i=0}^{\ell_{j-1}} $,
$j=1,2,3$, are flags of $ V_1,V_2,V_3 $ with dimension vectors $\mbfa',\mbfb',\mbfc'$, respectively, and $ F^{(4)} = ( F^{(4)}_i )_{i=0}^{\ell_4} $
is a flag of $V$ with dimension vector $ \mbfd $.
(In other words,
$ \dim F^{(1)}_i / F^{(1)}_{i-1} = a_i $ for all $i$, and similarly for $F^{(2)},F^{(3)},F^{(4)}$.)
In this way, an element of $\Xfv(\lambda)$
can be viewed as an object of $\Rep'\quiver$
of dimension vector $\lambda$,
so
$$
\Xfv(\lambda)\subset\Rep'\quiver.
$$
Note that, for example, the arrow $ \alpha_i $ is represented by the inclusion
$ F_i^{(1)} \hookrightarrow F_{i + 1}^{(1)} $.

\begin{proposition}\label{proposition:orbits.in.xfv.and.RepQ'}
\begin{penumerate}
\item
Every object in $ \Rep'\quiver $ of dimension vector $\lambda$
is isomorphic to an element of $ \Xfv(\lambda)$.
\item
Two elements $ F,F'\in\Xfv(\lambda) $ belong to the same $L$-orbit if and only if they are isomorphic
as representations of $\quiver$.
\end{penumerate}
\end{proposition}

\begin{proof}
(1) Let $\pi$ be an object of $\Rep'\quiver$ of dimension vector $\lambda$.
By virtue of (\ref{repQ'}), up to isomorphism we can assume that $\pi(D_{\ell_4})=V$, $\mathrm{Im}\,\pi(\alpha_{\ell_1-1})=V_1$,
$\mathrm{Im}\,\pi(\beta_{\ell_2-1})=V_2$,
and $\mathrm{Im}\,\pi(\gamma_{\ell_3-1})=V_3$.
By setting
$$
F_i^{(1)}:=\pi(\alpha_{\ell_1-1})\circ\cdots\circ\pi(\alpha_i)(\pi(A_i))\subset\mathrm{Im}\,\pi(\alpha_{\ell_1-1})=V_1,
$$
we get a flag $F^{(1)}=(F_i^{(1)})_{i=0}^{\ell_1-1}\in\Flags(\mbfa)$.
Proceeding similarly for $\mbfb$, $\mbfc$, $\mbfd$, we obtain an element $(F^{(1)},F^{(2)},F^{(3)},F^{(4)})\in\Xfv(\lambda)$ which is by construction isomorphic as a representation of $\quiver$ to the original representation $\pi$.

(2)
If $F'=g F$ for some $g\in L$, then for every subspace $F_i^{(j)}$ of $F$,
the automorphism $g:V\to V$ restricts to a well-defined isomorphism $g|_{F_i^{(j)}}:F_i^{(j)}\to F_i'^{(j)}$,
where $F_i'^{(j)}$ is the corresponding subspace of $F'$;
and this exactly implies that $F,F'$ are isomorphic as representations of $\quiver$.

Conversely, assume that $F$ and $F'$ are isomorphic as representations.
This means that there is a collection of linear isomorphisms
$\phi_i^{(j)}:F_i^{(j)}\to F_i'^{(j)}$ for all $(i,j)$,
which commute with the inclusions $F_i^{(j)}\subset F_{\ell_4}^{(4)}=V$ and $F_i'^{(j)}\subset F_{\ell_4}'^{(4)}=V$.
In other words, we obtain a linear automorphism $g=\phi_{\ell_4}^{(4)}:F_{\ell_4}^{(4)}=V\to F_{\ell_4}'^{(4)}=V$, i.e., $g\in \GL_n$,
such that $g(F_i^{(j)})=F_i'^{(j)}$ for all $i,j$.
In particular $g(V_j)=g(F_{\ell_j-1}^{(j)})=F_{\ell_j-1}'^{(j)}=V_j$ for all $j\in\{1,2,3\}$.
Thus $g\in L$, and $F'=g F$.
\end{proof}

This proposition tells us that $\Xfv(\lambda)$ will have finitely many $L$-orbits if and only if $\Rep'\quiver$ has finitely many isomorphism classes of objects of dimension vector $\lambda$.

By the Krull--Schmidt theorem, every object in $\Rep\quiver$ can be written as a direct sum of indecomposable objects, and the decomposition is unique up to isomorphisms and permutation of the summands.
Part (1) of the next lemma implies that the decomposition of an object of the subcategory $\Rep'\quiver$ as a sum of indecomposable objects in $\Rep\quiver$ is also a sum of indecomposable objects in $\Rep'\quiver$, so that the Krull--Schmidt theorem is also valid in $\Rep'\quiver$.

\begin{lemma}\label{lemma:quiver.category}
\begin{penumerate}
\item If $\pi$ is an object in $\Rep'\quiver$ and $\pi=\pi' \oplus \pi''$
is a decomposition in $\Rep\quiver$,
then $\pi'$ and $\pi''$ are actually objects in $\Rep'\quiver$.
\item If $\pi_1,\pi_2,\pi$ are objects in $\Rep'\quiver$ such that $\pi_1$ and $\pi_2$ are not isomorphic, then $\pi_1 \oplus \pi$ and $\pi_2 \oplus \pi$ are not isomorphic.
\end{penumerate}
\end{lemma}

\begin{proof}
(1) follows from the fact that injectivity of arrows and direct sums are preserved by decomposition.

(2) follows from Krull--Schmidt theorem for the category $ \Rep\quiver $ and (1).
\end{proof}

In the next lemma, we again fix a dimension vector $\lambda=(\mbfa,\mbfb,\mbfc,\mbfd)\in\Lambda$ and still consider the action of $L=\GL_{|\mbfa'|}\times \GL_{|\mbfb'|}\times \GL_{|\mbfc'|}$ on the double flag variety $\Xfv(\lambda)$.

\begin{lemma}\label{lemma:criterion.dense.orbit}
\begin{penumerate}
\item We have
\begin{equation}\label{dimension:Xabcd}
\dim \Xfv(\lambda)=\dim L-(\lambda|\lambda).
\end{equation}
\item The following implication holds:
\begin{equation}\label{eq:dense.orbit.implies.form.is.positive}
\text{$\Xfv(\lambda)$ has a dense $L$-orbit}\ \implies \ (\lambda|\lambda)\geq 1.
\end{equation}
In particular, if $(\lambda|\lambda)\leq 0$,
then the double flag variety $\Xfv(\lambda)$ has infinitely many orbits.
\item
If $ \lambda $ has a summand $ \mu \in \Lambda $ which satisfies $ (\mu|\mu)\leq 0$,
then the subcategory $ \Rep' \quiver $ contains infinitely many pairwise non-isomorphic objects of dimension vector $\lambda$, and therefore, $\Xfv(\lambda)$ has infinitely many $L$-orbits.
\end{penumerate}
\end{lemma}

\begin{proof}
(1) Indeed, we have
$$
\dim\GL_{|\mbfa'|}/Q_{\mbfa'}=\frac{1}{2}(|\mbfa'|^2-\sum_{i=1}^{\ell_1-1}a_i^2)
$$
and similarly for $\dim\GL_{|\mbfb'|}/Q_{\mbfb'}$ and $\dim\GL_{|\mbfc'|}/Q_{\mbfc'}$, while
$$
\dim\GL_{n}/P_{\mbfd}=\frac{1}{2}(n^2-\sum_{m=1}^{\ell_4}d_m^2)
$$
where $n:=|\lambda|$.
Hence, from (\ref{eq:general.Tits.form}) and (\ref{eq:def.of.Xabcd}), we obtain
\begin{eqnarray*}
    2\dim\Xfv(\lambda) & = & |\mbfa'|^2+|\mbfb'|^2+|\mbfc'|^2
    -\sum_{i=1}^{\ell_1-1}a_i^2
    -\sum_{j=1}^{\ell_2-1}b_j^2
    -\sum_{k=1}^{\ell_3-1}c_k^2
    +n^2-\sum_{m=1}^{\ell_4}d_m^2
    \\
    & = &
    |\mbfa'|^2+|\mbfb'|^2+|\mbfc'|^2
    -2(\lambda|\lambda)
    -n^2+a_{\ell_1}^2+b_{\ell_2}^2
    +c_{\ell_3}^2
    \\
    & = &
    |\mbfa'|^2+|\mbfb'|^2+|\mbfc'|^2-2(\lambda|\lambda)
    \\
    &&
    -(|\mbfa'|+|\mbfb'|+|\mbfc'|)^2+(|\mbfb'|+|\mbfc'|)^2+(|\mbfa'|+|\mbfc'|)^2
    +(|\mbfa'|+|\mbfb'|)^2
    \\
    & = &
    2|\mbfa'|^2+2|\mbfb'|^2+2|\mbfc'|^2
    -2(\lambda|\lambda).
\end{eqnarray*}
This completes the proof of (1).

(2)\
The claim follows from (\ref{dimension:Xabcd}) and the observation that the one-dimensional subgroup $\K^*I_n\subset L$ acts trivially on $\Xfv(\lambda)$.

(3)\
Suppose that $\lambda=\mu+\nu$ with $\mu\in\Lambda$ such that
$(\mu|\mu)\leq 0$ and $\nu\in\Lambda\cup\{0\}$.
By the claim (2) of the lemma,
combined with Proposition \ref{proposition:orbits.in.xfv.and.RepQ'},
there are infinitely many pairwise non-isomorphic objects $\pi_n$ ($n\in\mathbb{Z}$)
in $\Rep'\quiver$ with the same dimension vector $\mu$.
Choosing a further object $\pi'$ with dimension vector $\nu$ (take $\pi'=0$ if $\nu=0$ and, e.g., any point of the variety $\Xfv(\nu)$ if $\nu\not=0$), we get that $\pi_n\oplus\pi'$ are infinitely many objects in $\Rep'\quiver$ of dimension vector $\lambda$ which are pairwise non isomorphic (by Lemma \ref{lemma:quiver.category}\,(2)).
\end{proof}

Finally, based on a result of Kac \cite{Kac.Quiver.1980} and the Krull--Schmidt theorem, the following characterization holds.

\begin{theorem}\label{theorem:finiteness.criterion}
\begin{penumerate}
\item\label{theorem:finiteness.criterion:item:1}
Given $ \lambda=(\mbfa,\mbfb,\mbfc,\mbfd)\in\Lambda$, the double flag variety $ \Xfv(\lambda) $ has finitely many $L$-orbits if and only if every summand $\mu$ of $\lambda$ in $\Lambda$ satisfies $(\mu|\mu)\geq 1$.

\item\label{theorem:finiteness.criterion:item:2}
Moreover, if $ \lambda $ satisfies the condition in \eqref{theorem:finiteness.criterion:item:1}, there is a one-to-one correspondence
between the $L$-orbits of $\Xfv(\lambda)$
and the decompositions (up to permutation of the terms) $\lambda=\sum_{i=1}^k \mu_i$
with $k\geq 1$, $\mu_i\in\Lambda$, $(\mu_i|\mu_i)=1$.

\item\label{theorem:finiteness.criterion:item:3}
If $ \lambda $ satisfies the condition in \eqref{theorem:finiteness.criterion:item:1}
with the Tits form $(\lambda|\lambda)=1$,
then any representative of the open dense $ L $-orbit in $\Xfv(\lambda)$
is an indecomposable object
of $\Rep'\quiver$ of dimension vector $\lambda$.
\end{penumerate}
\end{theorem}

\begin{proof}
\eqref{theorem:finiteness.criterion:item:1}\
The implication $\Rightarrow$
follows from Lemma \ref{lemma:criterion.dense.orbit}\,(3).
For the reversed implication, according to Proposition \ref{proposition:orbits.in.xfv.and.RepQ'}, we have to show that the category $\Rep'\quiver\subset\Rep\quiver$ contains finitely many isomorphism classes of objects of dimension vector $\lambda$.
We use the following consequences of \cite[Theorem 2]{Kac.Quiver.1980}, which is valid in the larger category $\Rep\quiver$:
\begin{itemize}
    \item[\rm (a)] there is an indecomposable object in $\Rep\quiver$ of dimension vector $\mu$ only if $(\mu|\mu)\leq 1$;
    \item[\rm (b)] if $(\mu|\mu)=1$, then there is a unique isomorphism class of indecomposable objects in $\Rep\quiver$ of dimension vector $\mu$; and we fix a representative $\pi^\mu\in\Rep\quiver$ of this class.
\end{itemize}
The result in \cite[Theorem 3]{Kac.Quiver.1980} also tells us that, if $(\mu|\mu)\leq 0$, then $\Rep\quiver$ contains infinitely many indecomposable objects of dimension vector $\mu$, but we will not need this part of the result (we can use Lemma \ref{lemma:criterion.dense.orbit}\,(2) instead).

For $\lambda\in\Lambda$, we denote by $M(\lambda)\subset\Lambda$ the set of summands $\mu$ of $\lambda$ in $\Lambda$ and by $M^1(\lambda)\subset M(\lambda)$ the subset of summands $\mu$ satisfying $(\mu|\mu)=1$. Finally, let $\Sigma^1(\lambda)$ denote the set of all decompositions $\lambda=\mu_1+\ldots+\mu_k$ (up to permutation of the terms) with $k\geq 1$ and $\mu_i\in M^1(\lambda)$ for all $i$; in other words, the elements of $\Sigma^1(\lambda)$ are multisets $\{\mu_1,\ldots,\mu_k\}\subset M^1(\lambda)$ of sum $\lambda$. The sets $M(\lambda)$, $M^1(\lambda)$, and $\Sigma^1(\lambda)$ are clearly finite sets.

Finally, let $R'(\lambda)$ be the set of isomorphism classes
$[\pi]$ of objects in $\Rep'\quiver$ with dimension vector $\lambda$. We need to show that $R'(\lambda)$ is a finite set.

Let $\pi\in\Rep'\quiver$ be an object of dimension vector $\lambda$. By the Krull--Schmidt theorem, we have a unique decomposition $\pi=\pi_1\oplus\ldots\oplus\pi_k$ with indecomposable objects $\pi_i$ in $\Rep'\quiver$. Let $\mu_i$ be the dimension vector of $\pi_i$,
hence $\lambda=\mu_1+\ldots+\mu_k$. Every $\mu_i$
is a summand of $\lambda$ in $\Lambda$, so it satisfies on one hand $(\mu_i|\mu_i)\geq 1$ (by assumption) and on the other hand $(\mu_i|\mu_i)\leq 1$ (by property (a) above, since $\pi_i$ is indecomposable), thus $(\mu_i|\mu_i)=1$.
This shows that $\mu_i\in M^1(\lambda)$ for all $i$.
Whence a well-defined map
\begin{equation} \label{definition:Phi}
\Phi:R'(\lambda)\to\Sigma^1(\lambda),\quad [\pi]\mapsto \{\mu_1,\ldots,\mu_k\}.
\end{equation}
Moreover, for $\pi=\pi_1\oplus\ldots\oplus\pi_k$ mapped to $\{\mu_1,\ldots,\mu_k\}$ by $\Phi$ as before, Property (b) implies that
$\pi_i\cong \pi^{\mu_i}$ for all $i\in\{1,\ldots,k\}$, so
$[\pi]=[\pi^{\mu_1}\oplus\ldots\oplus \pi^{\mu_k}]$.
This shows that the map $\Phi$ is injective.
Therefore, $R'(\lambda)$ is a finite set and \eqref{theorem:finiteness.criterion:item:1} follows.

\eqref{theorem:finiteness.criterion:item:3}\
Here we assume that $\lambda$ satisfies the condition of (\ref{theorem:finiteness.criterion:item:1}) and $(\lambda|\lambda)=1$. By part (\ref{theorem:finiteness.criterion:item:1}), we know that $\Xfv(\lambda)$ has finitely many $L$-orbits thus there is a dense orbit $L\cdot F$ for some $ F \in \Xfv(\lambda) $.
Let $\pi=F$ be the object in $\Rep'\quiver$ corresponding to any representative of this dense orbit.
Arguing as in \cite[Proposition 3.2]{MWZ.1999},
let us prove that $\pi$ is \emph{Schur indecomposable}
(i.e., $\dim \Hom_{\Rep'\quiver}(\pi, \pi)=1$)
and in particular it is an indecomposable object of $\Rep'\quiver$.
To prove it, we first note that
\begin{equation*}
    \dim \Hom_{\Rep'\quiver}(\pi, \pi) = \dim \GL(V)\cap \Hom_{\Rep'\quiver}(\pi, \pi).
\end{equation*}
By the definition of $\Rep'\quiver$ and the fact that $\pi=F\in\Xfv(\lambda)$,
the stabilizer $ \Stab_{\GL(V)}(\pi)=\GL(V)\cap \Hom_{\Rep'\quiver}(\pi, \pi) $ is contained in $ L $.
Since $L\cdot F$ is dense in $ \Xfv(\lambda) $, we have
\begin{equation*}
    \dim \Stab_{\GL(V)}(\pi)
    = \dim \Stab_{L}(F) = \dim L - \dim \Xfv(\lambda) = (\lambda|\lambda) = 1
\end{equation*}
by \eqref{dimension:Xabcd},
which means that $ \pi $ is Schur indecomposable and $ \dim \pi = \lambda $ indeed.

\eqref{theorem:finiteness.criterion:item:2}
We return to a general $\lambda$ satisfying the condition of (\ref{theorem:finiteness.criterion:item:1}) and we need to show that the injective map $\Phi$ of (\ref{definition:Phi}) is also surjective.
Part (\ref{theorem:finiteness.criterion:item:3}) implies that for every $\mu\in M^1(\lambda)$ there is an indecomposable object in $\Rep'\quiver$ with dimension vector $\mu$; this in fact means that the indecomposable object $\pi^\mu$ (introduced in property (b) above) belongs to $\Rep'\quiver$ whenever $\mu\in M^1(\lambda)$.
Then for every $\{\mu_1,\ldots,\mu_k\}\in\Sigma^1(\lambda)$,
the isomorphism class $[\pi^{\mu_1}\oplus\ldots\oplus\pi^{\mu_k}]$ is a well-defined element of $R'(\lambda)$ which is an antecedent of $\{\mu_1,\ldots,\mu_k\}$ by $\Phi$. This completes the proof.
\end{proof}

\begin{example}
Let $ \lambda=(\mbfa,\mbfb,\mbfc,\mbfd)\in\Lambda(\bfell) $.
Even if there is a decomposition $\lambda=\sum \mu_i$ with $\mu_i\in\Lambda(\bfell)$ and $(\mu_i|\mu_i)=1$ for all $i$,
the double flag variety $ \Xfv(\lambda) $ may not have a dense $L$-orbit.

Indeed,
let $\lambda=((1,1,2),(1,1,2),(4),(1,1,1,1))$.
Then $(\lambda|\lambda)=0$, so that $\Xfv(\lambda)$ has no dense $L$-orbit.
However, $\lambda=\mu+\nu$ where
$$
\mu=((1,0,1),(1,0,1),(2),(1,0,1,0)),\quad \nu=((0,1,1),(0,1,1),(2),(0,1,0,1)),
$$
and we can see that $(\mu|\mu)=(\nu|\nu)=1$.
\end{example}

\section{A classification of double flag varieties of Levi type with finitely many orbits}\label{section:classification.proofs}

In this section, we complete the proof of Theorem \ref{T:main_theorem}. As explained in the introduction, it remains to treat part (C) of this theorem.
We apply the results in the previous section for the quiver $\quiver=\quiver(\bfell)$ of \eqref{eq:quiver} with four branches
to the present case where the rightmost branch of the quiver $ \quiver $ has length two.
Hence we consider quadruples of lengths of the form
$$
\bfell=(\ell_1,\ell_2,\ell_3,2)\quad\mbox{with}\quad\ell_1,\ell_2,\ell_3\geq 1.
$$
Then the semigroup $\Lambda(\bfell)$ consists of nonzero quadruples of dimension vectors of the form
\begin{equation}\label{form-lambda}
\lambda=(\mbfa,\mbfb,\mbfc,\mbfd)=((a_1,\ldots,a_{\ell_1}),(b_1,\ldots,b_{\ell_2}),(c_1,\ldots,c_{\ell_3}),(d_1,d_2))
\end{equation}
satisfying the conditions
\begin{equation}
\label{condition-Lambda}
 \sum_{i=1}^{\ell_1}a_i=\sum_{j=1}^{\ell_2}b_j=\sum_{k=1}^{\ell_3}c_k=d_1+d_2=\sum_{i=1}^{\ell_1-1}a_i+\sum_{j=1}^{\ell_2-1}b_j+\sum_{k=1}^{\ell_3-1}c_k=:|\lambda|.
\end{equation}
In this way we obtain the double flag variety of the form
$$
\Xfv(\lambda)=\Flags(\mbfa')\times\Flags(\mbfb')\times\Flags(\mbfc')\times\Grass(d;n)
$$
where $d:=d_1$ and $n:=d_1+d_2=|\lambda|$ (see \eqref{eq:def.of.Xabcd}).
Here as before we write $\mbfa'=(a_1,\ldots,a_{\ell_1-1})$ and write similarly $\mbfb'$ and $\mbfc'$. In this way $\mbfa',\mbfb',\mbfc'$ are compositions of the numbers $p:=|\mbfa'|$, $q:=|\mbfb'|$, $r:=|\mbfc'|$ respectively,
and $ (p, q, r) $ gives a composition of $n=|\lambda|$.

Note that the coefficients of the tuples in $\lambda$ are allowed to be $0$. If we remove from $\mbfa',\mbfb',\mbfc'$ the coefficients which are equal to $0$, then we get a new dimension vector $\hat\lambda$ which belongs to $\Lambda(\hat\bfell)$ for some $\hat\bfell=(\hat\ell_1,\hat\ell_2,\hat\ell_3,2)$ with $1\leq\hat\ell_i\leq\ell_i$ for $i\in\{1,2,3\}$,
but this new dimension vector $\hat\lambda$ gives rise to the same double flag variety $\Xfv(\hat\lambda)=\Xfv(\lambda)$ (up to immediate isomorphism).

We will use the following definition which actually allows us to compare the original $\lambda$ with other dimension vectors of different lengths. 
Here we call subsequence of $\mbfa$ (of length $m_1$) a sequence of the form $(a_{i_1},\ldots,a_{i_{m_1}})$ where $1\leq i_1<\ldots<i_{m_1}\leq \ell_1$, and use the same notion for $\mbfb$ and $\mbfc$.


\begin{definition} \label{D:partial.order.dim.vectors}
Let
$ \mbfm = (m_1, m_2, m_3, 2) $ with $m_1,m_2,m_3\geq 1$ be another quadruple of lengths and let $\mu=(\mbfe,\mbff,\mbfg,\mbfh)\in\Lambda(\mbfm)$ be a dimension vector in the corresponding semigroup, thus $\mbfh=(h_1,h_2)$.

We write $\mu\leq \lambda$ if
$h_i\leq d_i$ for $i\in\{1,2\}$ and
there are subsequences of $\mbfa$, $\mbfb$, $\mbfc$ which are pointwise bigger than or equal to $\mbfe$, $\mbff$, $\mbfg$.
\end{definition}

\skipover{
Let us denote
$ \Lambda = \coprod_{\bfell} \Lambda(\bfell) $,
the disjoint union of all $ \Lambda(\bfell) \; (\bfell = (\ell_1,\ell_2,\ell_3)) $ satisfying
$ \ell_1, \ell_2, \ell_3 \geq 1 $.

For $\lambda \in \Lambda(\ell_1,\ell_2,\ell_3) $, let $\hat\lambda$ be the subsequence obtained from $\lambda$ by removing all the coefficients among $a_1,\ldots,a_{\ell_1-1}, b_1,\ldots,b_{\ell_2-1},c_1,\ldots,c_{\ell_3-1}$ which are equal to $0$'s.
Thus $\hat\lambda\in\Lambda(\hat{\ell}_1,\hat{\ell}_2,\hat{\ell}_3)$ for some lengths $\hat{\ell}_1,\hat{\ell}_2,\hat{\ell}_3$ such that $\hat{\ell}_i\leq\ell_i$ for all $i\in\{1,2,3\}$.

We define a partial order on the whole $ \Lambda $ in the following way.

\begin{definition}
Let $\lambda\in \Lambda(\ell_1,\ell_2,\ell_3)$ and
$\mu\in \Lambda(m_1,m_2,m_3)$, and we use the notation of (\ref{form-lambda}),
with the letters $a'_i,b'_j,c'_k,d'_1,d'_2$ for denoting the coefficients of $\mu$.
We write $\mu\leq\lambda$ if
$d'_1\leq d_1$, $d'_2\leq d_2$, and
there are subsequences
$(a_{i_1},\ldots,a_{i_{m_1-1}})$,
$(b_{j_1},\ldots,b_{j_{m_2-1}})$,
$(c_{k_1},\ldots,c_{k_{m_3-1}})$
of $(a_1,\ldots,a_{\ell_1-1})$,
$(b_1,\ldots,b_{\ell_2-1})$,
$(c_1,\ldots,c_{\ell_3-1})$,
which are pointwise bigger than or equal to
$(a'_1,\ldots,a'_{m_1-1})$,
$(b'_1,\ldots,b'_{m_2-1})$,
$(c'_1,\ldots,c'_{m_3-1})$, respectively.
\end{definition}}

If $\mu\leq\lambda$, then in particular $m_i\leq\ell_i$ for $i\in\{1,2,3\}$.
Also we have $|\mbfe'|\leq |\mbfa'|$, $|\mbff'|\leq|\mbfb'|$, $|\mbfg'|\leq|\mbfc'|$, hence
$e_{m_1}=|\mbff'|+|\mbfg'|\leq |\mbfb'|+|\mbfc'|=a_{\ell_1}$, and similarly $f_{m_2}\leq b_{\ell_2}$ and $g_{m_3}\leq c_{\ell_3}$.
Clearly, $\leq$ is an order relation on $\bigcup_{m_1,m_2,m_3\geq 1} \Lambda(\mbfm)$.

If $\mbfm=\bfell$ then the relation $\mu\leq\lambda$ means that $\lambda$ is pointwise bigger than or equal to $\mu$ (no need to extract subsequences in this case), which thus just means that $\mu$ is a summand of $\lambda$ in $\Lambda(\bfell)=\Lambda(\mbfm)$.
The following lemma, which is immediate from the definition above, generalizes this observation.

\begin{lemma}\label{L-summand}
\newcommand{\bfemm}{\boldsymbol{m}}
Let $\lambda\in\Lambda(\bfell)$
and $\mu=(\mbfe,\mbff,\mbfg,\mbfh)\in\Lambda(\bfemm)$.
We have $\mu\leq \lambda$ if and only if there exists a summand $\tilde\mu$ of $\lambda$ in $\Lambda(\bfell)$
which is obtained from $\mu$ by inserting some $0$'s in $\mbfe,\mbff,\mbfg$
(so that one has in particular $(\tilde\mu|\tilde\mu)=(\mu|\mu)$).
\end{lemma}

Put
\begin{eqnarray*}
 & \lambda^1=((1^2,4),(1^2,4),(1^2,4),(3^2)), \\
 & \lambda^2=((2,6),(1^3,5),(1^3,5),(4^2)), \\
 & \lambda^3=((3,9),(2^2,8),(1^5,7),(6^2)),
\end{eqnarray*}
where (as it is usual) the notation $i^k$ stands for a sequence of $k$ coefficients equal to $i$.
One can see that
\begin{equation}
\label{Tits-form-lambda-i}
(\lambda^1|\lambda^1)=(\lambda^2|\lambda^2)=(\lambda^3|\lambda^3)=0.
\end{equation}

We are now ready to state the main result of this section.

\begin{theorem}\label{Thm:classification.finite.type.MFV}
Let $\lambda\in\Lambda(\bfell)$.
Write
\begin{eqnarray*}
\lambda=(\mbfa,\mbfb,\mbfc,\mbfd) & = & ((a_1,\ldots,a_{\ell_1}),(b_1,\ldots,b_{\ell_2}),(c_1,\ldots,c_{\ell_3}),(d_1,d_2)) \\
 & = & ((\mbfa',a_{\ell_1}),(\mbfb',b_{\ell_2}),(\mbfc',c_{\ell_3}),(d_1,d_2)).
\end{eqnarray*}
The following conditions are equivalent.
\begin{penumerate}
\item $\Xfv(\lambda)$ is of finite type;
\item $\lambda\not\geq\lambda^1$, $\lambda\not\geq\lambda^2$, and $\lambda\not\geq\lambda^3$,
for any permutation of $\mbfa,\mbfb,\mbfc$;
\item $\lambda$ fits in one of the following cases, up to permuting $\mbfa,\mbfb,\mbfc$ and up to removing the coefficients of $\mbfa',\mbfb',\mbfc'$ which are equal to $0$:
\begin{itemize}
\item[\rm (o)] $\ell_1=1$;
\item[\rm (a)] $\min\{d_1,d_2\}\leq 2$;
\item[\rm (b)] $\min\{d_1,d_2\}=3$ and $\ell_1=2$;

\item[\rm (c)] $\min\{d_1,d_2\}\in\{4,5\}$, $\ell_1=2$, and $\ell_2=3$;

\item[\rm (d)] $\ell_1=2$ and one of the following conditions occurs:
\begin{itemize}
\item[\rm (d1)] $a_1=1$;
\item[\rm (d2)] $a_1=2$ and $\ell_2=3$;
\item[\rm (d3)] $\ell_2=2$ or $ (\ell_2=3 \; \text{ and } \, \min\{b_1,b_2\}=1) $;
\item[\rm (d4)] $\ell_2=3$ and $\ell_3\leq 5$.
\end{itemize}
\end{itemize}
%
%
%

\end{penumerate}
\end{theorem}

\begin{remark}
Table \ref{table:finite.type.m=3} can be derived from this result.
Case (o) of the theorem, where $\ell_1=1$, corresponds to the situation where $|\mbfa'|=0$, i.e., the Levi subgroup $L=\GL_{|\mbfa'|}\times \GL_{|\mbfb'|}\times \GL_{|\mbfc'|}$ has actually two blocks. This case does not appear in Table \ref{table:finite.type.m=3} (which concerns Levi subgroups with three blocks) but corresponds to the first case referenced in Table \ref{table:finite.type.m=2}.
\end{remark}

\begin{proof}[Proof  of Theorem \ref{Thm:classification.finite.type.MFV}]
(1)$\Rightarrow$(2): If $\lambda\geq\lambda^i$ for some $i\in\{1,2,3\}$, then by Lemma \ref{L-summand}
there is a summand $\mu$ of $\lambda$ in $\Lambda(\bfell)$ such that $(\mu|\mu)=(\lambda^i|\lambda^i)=0$,
and Theorem \ref{theorem:finiteness.criterion} implies that $\Xfv(\lambda)$ is not of finite type.

\bigskip

(2)$\Rightarrow$(3):
Let $\hat\lambda$ be the dimension vector obtained from $\lambda$
by removing all $0$'s from $\mbfa',\mbfb',\mbfc'$. Then a fortiori $\hat\lambda$ satisfies condition (2). Since it suffices to show that $\hat\lambda$ satisfies (3), we can actually assume that $\lambda=\hat\lambda$, i.e., that all coefficients of $\mbfa',\mbfb',\mbfc'$ are positive.

We set $p=|\mbfa'|$, $q=|\mbfb'|$, $r=|\mbfc'|$, and $n=|\lambda'|=p+q+r$, so that $a_{\ell_1}=q+r$,
$b_{\ell_2}=p+r$, and $c_{\ell_3}=p+q$.
Also, up to permuting $\mbfa,\mbfb,\mbfc$, we can assume that
$1\leq \ell_1\leq \ell_2\leq \ell_3$.
If $\ell_1=1$ or $\min\{d_1,d_2\}\leq 2$, there we are in case (o) or (a). In the following we assume $\ell_1\geq 2$ and $\min\{d_1,d_2\}\geq 3$,
and show that $\lambda$ fits in one of the cases (b)--(d) of part (3) of the theorem.

\medskip\noindent
{\it Case 1:} $\min\{d_1,d_2\}=3$.

We then have $\mbfd=(d_1,d_2)\geq(3,3)$.
If $\ell_1\geq 3$, then $\mbfa',\mbfb',\mbfc'$ all contain at least two (positive) coefficients. This yields $p,q,r\geq 2$ which implies that $a_{\ell_1}=q+r\geq 4$ and similarly $b_{\ell_2}\geq 4$ and $c_{\ell_3}\geq 4$. But then we get $\lambda\geq((1,1,4),(1,1,4),(1,1,4),(3,3))=\lambda^1$ which contradicts the assumption made in part (2) of the theorem. Hence $\ell_1=2$, so that we are in case (b).

\medskip\noindent
{\it Case 2:} $\min\{d_1,d_2\}\in\{4,5\}$.

Then $\mbfd\geq (4,4)\geq (3,3)$. As in Case 1, we must have $\ell_1=2$, hence $\mbfa=(p,n-p)$.
If $p\geq 2$ and $\ell_2\geq 4$, then $q,r\geq 3$, which
yields $n-p=q+r\geq 6$ thus $\mbfa\geq (2,6)$, while $b_{\ell_2}=p+r\geq 5$, so $\mbfb\geq (1^3,5)$
and similarly for $\mbfc$; altogether we obtain $\lambda\geq ((2,6),(1^3,5),(1^3,5),(4^2))=\lambda^2$, a contradiction with the assumption made in (2).
Consequently, we must have $p=1$ or $\ell_2\leq 3$, so we are in case (d1), (c), or (d3) of the theorem.

\medskip
\noindent
{\it Case 3:} $\min\{d_1,d_2\}\geq 6$.

Then $\mbfd\geq (6,6)$. As in Case 2, we must have $\ell_1=2$, and $p=a_1=1$ or $\ell_2\leq 3$.
We will assume $p\geq 3$ and $\ell_2=3$, otherwise we are directly in case (d1), (d2) or (d3) of the theorem and we are done. Thus $\mbfa=(p,n-p)$ with $p\geq 3$, $\mbfb=(b_1,b_2,n-q)$,
and $\mbfc=(c_1,\ldots,c_{\ell_3-1},n-r)$.

If $\min\{b_1,b_2\}\geq 2$ and $\ell_3\geq 6$,
then $q=b_1+b_2\geq 4$ and $r\geq 5$;
thus $n-p=q+r\geq 9$, $n-q=p+r\geq 8$, and $n-r=p+q\geq 7$,
whence $\lambda\geq ((3,9),(2^2,8),(1^5,7),(6^2))=\lambda^3$, in contradiction with (2).
Therefore, $\min\{b_1,b_2\}=1$ or $\ell_3\leq 5$, and we are in case (d3) or (d4) of the theorem.
The proof of (2)$\Rightarrow$(3) is complete.

\bigskip

(3)$\Rightarrow$(1): It is easy to see that if $\lambda$ fits in one of the cases (o), (a)--(d) of the theorem, then so does every summand $\mu$ of $\lambda$. Therefore, by virtue of Theorem \ref{theorem:finiteness.criterion}, for showing the desired implication, it is sufficient to check that
$$
(\lambda|\lambda)\geq 1\quad\mbox{whenever $\lambda$ fits in one of the cases (o), (a)--(d)}.
$$
In the calculations below, in each case, we also determine the $\lambda$'s for which the equality $(\lambda|\lambda)=1$ holds.
In the next section, thanks to Theorem \ref{theorem:finiteness.criterion}, this will lead to a parametrization of the orbits of $\Xfv(\lambda)$ when this double flag variety is of finite type.

Up to removing the zero coefficients, we can assume that all the coefficients of $\mbfa',\mbfb',\mbfc'$ are nonzero.
Set
$$
n=|\lambda|,\quad p=\sum_{i=1}^{\ell_1-1}a_i,\quad q=\sum_{j=1}^{\ell_2-1}b_j, \quad r=\sum_{k=1}^{\ell_3-1}c_k,\quad \mbfd=(d,n-d).
$$
Thus $n=p+q+r$.
Also there is no loss of generality in assuming that
$\ell_1\leq \ell_2\leq \ell_3$ and that
every sequence $\mbfa',\mbfb',\mbfc',\mbfd$ is nondecreasing,
in particular $d\leq n-d$ and
thus $d\leq\frac{n}{2}$.
We have
\begin{eqnarray}
2(\lambda|\lambda)
&=& \sum_{i=1}^{\ell_1-1}a_i^2+(n-p)^2+
\sum_{j=1}^{\ell_2-1}b_j^2+(n-q)^2+\sum_{k=1}^{\ell_3-1}c_k^2
+(n-r)^2+d^2+(n-d)^2-2n^2 \nonumber \\
&=& p^2+q^2+r^2+\sum_{i=1}^{\ell_1-1}a_i^2+
\sum_{j=1}^{\ell_2-1}b_j^2+\sum_{k=1}^{\ell_3-1}c_k^2+2d^2-2nd.
\label{3.4}
\end{eqnarray}
We note that
\begin{equation}
\sum_{i=1}^{\ell_1-1}a_i^2\geq \sum_{i=1}^{\ell_1-1}a_i=p \quad
\mbox{with equality if and only if $a_1=\ldots=a_{\ell_1-1}=1$;}
\label{3.5}
\end{equation}
in the same way $\sum_{j=1}^{\ell_2-1}b_j^2\geq q$ and $\sum_{k=1}^{\ell_3-1}c_k^2\geq r$,
with similar case of equality. Moreover,
\begin{eqnarray}
& p^2+q^2+r^2\geq\dfrac{n^2}{3}\quad \mbox{with equality if and only if $p=q=r=\dfrac{n}{3}$;}
\label{3.6}
\\
& \mbox{the map $x\mapsto 2x^2-2nx$ is decreasing on $[0,\frac{n}{2}]$,
hence $2d^2-2nd\geq -\dfrac{n^2}{2}$.}
\label{3.7}
\end{eqnarray}
These observations lead to the estimate
$$
2(\lambda|\lambda)\geq n-\frac{n^2}{6},\quad \mbox{hence $(\lambda|\lambda)>0$ whenever $n\leq 5$.}
$$
Moreover, it is straightforward to see that, for $n\leq 5$, the equality $(\lambda|\lambda)=1$ holds if and only if $\lambda$ is 
one of the tuples
\begin{eqnarray*}
 & ((1),(1),(1,0),(0,1)),\quad
((2),(1,1),(1,1),(1,1)),\quad
((1,2),(1,2),(1,2),(1,2)),\\
&
((1,3),(1,3),(1,1,2),(2,2)),\quad
((1,4),(1,1,3),(1,1,3),(2,3)).
\end{eqnarray*}
We now assume that $n\geq 6$ and consider separately the cases (o), (a)--(d).

\bigskip
\noindent
(o) $\ell_1=1$, i.e., $\mbfa=(n)$ and $p=0$:

Combining (\ref{3.4}), (\ref{3.5}) (applied to $\mbfb$ and $\mbfc$), and (\ref{3.7}) leads to
$$
2(\lambda|\lambda)\geq q^2+r^2+q+r-\frac{n^2}{2}\geq n\geq 6
$$
since $q+r=n$ and thus $q^2+r^2\geq \frac{n^2}{2}$.
Hence $(\lambda|\lambda)>1$; in particular $(\lambda|\lambda)>0$.

\bigskip
\noindent
(a) $d\leq 2$:

Using (\ref{3.4})--(\ref{3.7}), we find that
$$
2(\lambda|\lambda)\geq \frac{n^2}{3}+n+(8-4n)=\frac{n^2}{3}-3n+8\geq \frac{6^2}{3}-3\cdot 6+8=2,
$$
hence $(\lambda|\lambda)>0$.

Moreover, the equality $(\lambda|\lambda)=1$ holds if and only if
$n=6$, $d=2$, $p=q=r=\frac{n}{3}=2$, and all the coefficients of $\mbfa',\mbfb',\mbfc'$ are equal to $1$,
i.e., $\lambda=((1,1,4),(1,1,4),(1,1,4),(2,4))$.

\bigskip
\noindent
(b) $d=3$ and $\ell_1=2$; thus $\mbfa=(p,n-p)$:

Using (\ref{3.4}), (\ref{3.5}), and (\ref{3.7}), we see that
$$
2(\lambda|\lambda)\geq 2p^2+q^2+r^2+q+r+18-6n.
$$
Using Lagrange multipliers rule, the minimum of the map $(x,y,z)\mapsto 2x^2+y^2+z^2+y+z$
on the set $\{(x,y,z)\in\mathbb{R}^3:x+y+z=n\}$
(which exists since we are minimizing a coercive quadratic function on an affine space)
is attained at a point $(x,y,z)$ such that
$$
\left\{\begin{array}{ll}
4x & = \alpha \\ 2y+1 & =\alpha \\
2z+1 & =\alpha
\end{array}\right.
$$
for some $\alpha\in\mathbb{R}$.
Then $n=x+y+z=\frac{\alpha}{4}+\frac{\alpha-1}{2}+\frac{\alpha-1}{2}=\frac{5\alpha-4}{4}$, so $\alpha=\frac{4n+4}{5}$.
Whence $(x,y,z)=(\frac{n+1}{5},\frac{4n-1}{10},\frac{4n-1}{10})$. This yields
$$
2(\lambda|\lambda) \geq 2 \, \Bigl(\frac{n+1}{5}\Bigr)^2+2\, \Bigl(\frac{4n-1}{10}\Bigr)^2
+2\, \frac{4n-1}{10} + 18-6n=\frac{4n^2-52n+179}{10}>0,
$$
hence $(\lambda|\lambda)>0$.

Moreover, the same estimate tells us more precisely that
$(\lambda|\lambda)>1$ whenever $n\geq 9$. Also one can see that, for $n\in\{6,7,8\}$,
the only tuples $\lambda$ for which the equality $(\lambda|\lambda)=1$ holds
are
$$
\begin{array}{lll}
((1,5),(1^2,4),(1^3,3),(3,3)), & ((2,4),(1^2,4),(1^2,4),(3,3)), &
((1,6),(1^3,4),(1^3,4),(3,4)), \\
((2,5),(1^2,5),(1^3,4),(3,4)), &
((2,6),(1^3,5),(1^3,5),(3,5)).
\end{array}
$$

\medskip

In the remaining cases, we have $\ell_1=2$, hence we can assume that $d\geq 4$ and thus $n\geq 8$ (otherwise we are in case (a) or (b)).

\bigskip
\noindent
(c) $d\in\{4,5\}$, $\ell_1=2$, and $\ell_2=3$;
thus $\mbfa=(p,n-p)$ and $\mbfb=(b,q-b,n-q)$:

We have $b\leq q-b$, hence $b\leq\frac{q}{2}$.
Note that $b^2+(q-b)^2\geq \lfloor\frac{q}{2}\rfloor^2+ \lceil\frac{q}{2}\rceil^2$ with equality if and only if $b=\lfloor\frac{q}{2}\rfloor$.
Also if $n=8$ then $d=4$ and $2d^2-2nd=-32$;
if $n=9$ then $d=4$ and $2d^2-2nd=-40=50-10n$;
by (\ref{3.7}), if $n\geq 10$ then $2d^2-2nd\geq 50-10n$ with equality if and only if $d=5$.

If $n=8$, then (\ref{3.4})--(\ref{3.5}) yield
$2(\lambda|\lambda)\geq 2p^2+2q^2-2\lfloor\frac{q}{2}\rfloor \lceil\frac{q}{2}\rceil+r^2+r-32$ with equality if and only if $b=\lfloor\frac{q}{2}\rfloor$
and all the coefficients in $\mbfc'$ are $1$.
By considering the various triples $(p,q,r)$ with $p+q+r=8$, it turns out that we always have $(\lambda|\lambda)\geq 1$ with equality if and only if
$\lambda=((2,6),(1,1,6),(1^4,4),(4,4))$ or
$\lambda=((2,6),(1,2,5),(1^3,5),(4,4))$. 

For $n\geq 9$, from (\ref{3.4})--(\ref{3.5}), we write
$$
2(\lambda|\lambda)
\geq 2p^2+q^2+\Big\lfloor \frac{q}{2} \Big\rfloor^2+\Big\lceil \frac{q}{2} \Big\rceil^2+r^2+r+50-10n
\geq 2p^2+\frac{3}{2}q^2+r^2+r+50-10n.
$$
Here again we use Lagrange multipliers rule to determine the minimum of the map $(x,y,z)\mapsto 2x^2+\frac{3}{2}y^2+z^2+z$
on the set $\{(x,y,z)\in\mathbb{R}^3:x+y+z=n\}$. It
is attained at a point $(x,y,z)$ such that
$$
\left\{\begin{array}{ll}
4x & = \alpha \\ 3y & =\alpha \\
2z+1 & =\alpha
\end{array}\right.
$$
for some $\alpha\in\mathbb{R}$.
Then $n=x+y+z=\frac{\alpha}{4}+\frac{\alpha}{3}+\frac{\alpha-1}{2}=\frac{13\alpha-6}{12}$, so $\alpha=\frac{12n+6}{13}$.
Whence
\begin{eqnarray*}
2(\lambda|\lambda) & \geq &   2\Big(\frac{6n+3}{26}\Big)^2+\frac{3}{2}\Big(\frac{4n+2}{13}\Big)^2+\Big(\frac{12n-7}{26}\Big)^2+\frac{12n-7}{26}+50-10n \\
 & = & \frac{24n^2-496n+2593}{52}>0,
\end{eqnarray*}
so $(\lambda|\lambda)>0$.

From this estimate, we get in addition that $(\lambda|\lambda)>1$ whenever $n\geq 13$.
If $9\leq n\leq 12$, then it can be checked that the only tuples $\lambda$ such that $(\lambda|\lambda)=1$ are
$$
\begin{array}{lll}
((2,7),(1,2,6),(1^4,5),(4,5)), & ((2,8),(1,2,7),(1^5,5),(5,5)), & ((2,8),(2,2,6),(1^4,6),(5,5)),\\
((3,7),(1,2,7),(1^4,6),(5,5)), & ((2,9),(2,2,7),(1^5,6),(5,6)), & ((3,8),(1,2,8),(1^5,6),(5,6)),\\
((3,8),(2,2,7),(1^4,7),(5,6)), & ((3,9),(2,2,8),(1^5,7),(5,7)).
\end{array}
$$

\bigskip
\noindent
(d1) $\ell_1=2$ and $a_1=1$, thus $\mbfa=(1,n-1)$:

Here we have $p=1$ and so $q+r=n-1$. From \eqref{3.4} and \eqref{3.5} we obtain that
$$
2(\lambda|\lambda)\geq 1+q^2+r^2+1+q+r+(2d^2-2nd)=(n+1)+(q^2+r^2)+(2d^2-2nd)
$$
with equality if and only if all the coefficients in $\mbfb'$ and $\mbfc'$ are $1$.

If $n=2k+1$ is odd, we have
$$
2(\lambda|\lambda)\geq (n+1)+(2k^2)+(2k^2-2nk)=2
$$
with equality if and only if, moreover, $q=r=d=k$.

If $n=2k$ is even, then we get
$$
2(\lambda|\lambda)\geq (n+1)+((k-1)^2+k^2)+(2k^2-2nk)=2
$$
with equality if and only if, in addition to the fact that all coefficients in $\mbfb',\mbfc'$ are $1$, we have $d=k$ and $(q,r)=(k-1,k)$ (recall that $\ell_2\leq \ell_3$ by assumption).

In all the cases we get $(\lambda|\lambda)>0$. Moreover, the equality $(\lambda|\lambda)=1$ holds if and only if
$\lambda=((1,2k),(1^k,k+1),(1^k,k+1),(k,k+1))$ or
$\lambda=((1,2k-1),(1^{k-1},k+1),(1^k,k),(k,k))$ for some $k$.



\bigskip
\noindent
(d2) $\ell_1=2$, $a_1=2$, and $\ell_2=3$; thus $\mbfa=(2,n-2)$ and $\mbfb=(b,q-b,n-q)$:

We can assume that $d\geq 6$ and so $n\geq 12$
(otherwise, we are in case (c)).
By (\ref{3.4})--(\ref{3.5}) and (\ref{3.7}), we have
$$
2(\lambda|\lambda) \geq 8+\frac{3}{2}q^2+r^2+r-\frac{n^2}{2}
$$
with equality if and only if $b=\frac{q}{2}$, all coefficients in $\mbfc'$ are $1$, and $d=\frac{n}{2}$.
Using that $q=n-r-2$, we infer that
$$
2(\lambda|\lambda)\geq\frac{5}{2}r^2+7r-6n-3nr+n^2+14
= \Bigr( n-\frac{3r+6}{2} \Bigl)^2+\frac{(r-4)^2}{4}+1 > 0.
$$
This shows that $(\lambda|\lambda)\geq 1$ for all $\lambda$; moreover, the equality holds if and only if $r=6$, $n=12$, and in fact $\lambda=((2,10),(2,2,8),(1^6,6),(6,6))$.

\bigskip
\noindent
(d3)\  $\ell_1=2$ and ($\ell_2=2$ or $\ell_2=3$ with $\min\{b_1,b_2\}=1$),
i.e., $\mbfa=(p,n-p)$ and $\mbfb$ is $(q,n-q)$ or $(1,q-1,n-q)$:

We can deal with all situations simultaneously by supposing that $\mbfb=(b,q-b,n-q)$ with $b\in\{0,1\}$.
Also we can assume that $p,q,r\geq 2$ (otherwise, we are in case (d1)).
By (\ref{3.4}), (\ref{3.5}), and (\ref{3.7}), we have
\begin{equation}
\label{d3:estimate1}
2(\lambda|\lambda) \geq 2p^2+2q^2-2q+r^2+r+2-2\Big\lfloor\frac{n}{2}\Big\rfloor \Big\lceil\frac{n}{2}\Big\rceil
\end{equation}
with equality if and only if
$b=1$, all the coefficients in $\mbfc'$ are $1$, and $d=\lfloor\frac{n}{2}\rfloor$.

In the present case (d3), instead of relying on Lagrange multipliers rule like in previous cases, we prove the inequality $(\lambda|\lambda)>0$, in fact $(\lambda|\lambda)\geq 1$, through a different, more precise analysis, which will also allow us to characterize the cases of equality
(as we shall see, there are infinitely many tuples $\lambda$ with $(\lambda|\lambda)=1$).

%

Note that $p+q+r=n$ hence $p=n-r-q$. We aim to minimize the quantity $2p^2+2q^2-2q$ while $n$ and $r$ are fixed. To this end, we define
\begin{eqnarray*}
f(q) & = & p^2+q^2-q \\
 & = & 2q^2-(2(n-r)+1)q+(n-r)^2.
\end{eqnarray*}
The function $f(x)=2x^2-(2(n-r)+1)x+(n-r)^2$ attains its minimum at $x=\frac{n-r}{2}+\frac{1}{4}$. Hence, since $q$ must be an integer, the quantity $f(q)$ attains its minimum at $q=\frac{n-r}{2}$ if $n-r$ is even, or at $q=\frac{n-r+1}{2}$ if $n-r$ is odd, and in both cases the minimal value is
$$
f\Big(\frac{n-r}{2}\Big)=f\Big(\frac{n-r+1}{2}\Big)=\frac{(n-r)^2}{2}-\frac{n-r}{2}.
$$
Whence
\begin{equation}\label{minoration1}
2(\lambda|\lambda)\geq (n-r)^2-(n-r)+r^2+r+2-2\Bigl\lfloor\frac{n}{2}\Bigr\rfloor \Bigl\lceil\frac{n}{2}\Bigr\rceil
\end{equation}
with equality if and only if we have
\begin{equation}
\label{condition3}
q=p=\frac{n-r}{2}\ \mbox{if $n-r$ is even},\quad
q=\frac{n-r+1}{2}=p+1\ \mbox{if $n-r$ is odd},
\end{equation}
in addition to the conditions written after (\ref{d3:estimate1}).

The right-hand side of (\ref{minoration1}) can be rewritten as
$$
g(r):=2r^2-2(n-1)r+n^2-n+2-2\Bigl\lfloor\frac{n}{2}\Bigr\rfloor \Bigl\lceil\frac{n}{2}\Bigr\rceil.
$$
From here, we distinguish two cases.

\medskip
\noindent
{\it Case 1:} $n=2k+1$ is odd.

In this case, the minimum of $g(r)$ (while $r$ runs over integers) is attained at $r=\frac{n-1}{2}=k$ and the value of this minimum is
$$
g\Big(\frac{n-1}{2}\Big)=\frac{(n-1)^2}{2}-(n-1)^2+n^2-n+2-\frac{(n-1)(n+1)}{2}=2.
$$
In this case, we finally obtain that
$(\lambda|\lambda)\geq 1$ with equality if and only if we have $r=\frac{n-1}{2}$ in addition to (\ref{condition3}) and the conditions of equality in (\ref{d3:estimate1}).
Note that we then have $n-r=k+1$, hence $n-r$ is even when $k$ is odd and it is odd when $k$ is even.
Hence the equality $(\lambda|\lambda)=1$ holds if and only if
$$
\lambda = \left\{
\begin{array}{ll}
\displaystyle
\Big(\Big(\frac{k+1}{2},\frac{3k+1}{2}\Big),\Big(1,\frac{k-1}{2},\frac{3k+1}{2}\Big),(1^k,k+1),(k,k+1)\Big) & \mbox{if $k$ is odd}, \\[2mm]
\displaystyle
\Big(\Big(\frac{k}{2},\frac{3k}{2}+1\Big),\Big(1,\frac{k}{2},\frac{3k}{2}\Big),(1^k,k+1),(k,k+1)\Big) & \mbox{if $k$ is even}.
\end{array}
\right.
$$

\medskip
\noindent
{\it Case 2:} $n=2k$ is even.

In this case, the minimum of $g(r)$ (while $r$ runs over integers) is attained at $r=\frac{n}{2}=k$ and at $r=\frac{n}{2}-1=k-1$. The value of this minimum is
$$
g\Big(\frac{n}{2}\Big)=g\Big(\frac{n}{2}-1\Big)=\frac{n^2}{2}-n(n-1)+n^2-n+2-\frac{n^2}{2}=2.
$$
Here we therefore obtain that
$(\lambda|\lambda)\geq 1$ with equality if and only if we have $r\in\{\frac{n}{2},\frac{n}{2}-1\}$ in addition to (\ref{condition3}) and the conditions for equality in (\ref{d3:estimate1}).
We conclude that the equality $(\lambda|\lambda)=1$ holds if and only if
\begin{eqnarray*}
& \displaystyle \lambda=\Big(\Big(\frac{k-1}{2},\frac{3k+1}{2}\Big),\Big(1,\frac{k-1}{2},\frac{3k-1}{2}\Big),(1^k,k),(k,k)\Big) \\[2mm]
\mbox{or} & \displaystyle \lambda=\Big(\Big(\frac{k+1}{2},\frac{3k-1}{2}\Big),\Big(1,\frac{k-1}{2},\frac{3k-1}{2}\Big),(1^{k-1},k+1),(k,k)\Big)
\end{eqnarray*}
when $k$ is odd, and
\begin{eqnarray*}
& \displaystyle \lambda=\Big(\Big(\frac{k}{2},\frac{3k}{2}\Big),\Big(1,\frac{k}{2}-1,\frac{3k}{2}\Big),(1^k,k),(k,k)\Big) \\[2mm]
\mbox{or} & \displaystyle \lambda=\Big(\Big(\frac{k}{2},\frac{3k}{2}\Big),\Big(1,\frac{k}{2},\frac{3k}{2}-1\Big),(1^{k-1},k+1),(k,k)\Big)
\end{eqnarray*}
when $k$ is even.

\bigskip
\noindent
(d4)\  $\ell_1=2$, $\ell_2=3$, and $\ell_3\leq 5$; thus
$\mbfa=(p,n-p)$, $\mbfb=(b,q-b,n-q)$, and $\mbfc=(c_1,c_2,c_3,c_4,n-r)$
with $c_1,c_2,c_3,c_4\geq 0$ such that $c_1+c_2+c_3+c_4=r$:

We can assume that $d\geq 6$ hence $n\geq 12$, otherwise we are in case (a), (b), or (c). Moreover, $p,r\geq 3$, $q\geq 4$,
otherwise we are in case (d1), (d2), or (d3).
Easily we have
$c_1^2+c_2^2+c_3^2+c_4^2\geq 4\cdot(\frac{r}{4})^2=\frac{r^2}{4}$
and $b^2+(q-b)^2\geq\frac{q^2}{2}$. Using also (\ref{3.4}) and (\ref{3.7}), we get
$$
2(\lambda|\lambda)\geq 2p^2+\frac{3}{2}q^2+\frac{5}{4}r^2-\frac{n^2}{2}.
$$
Once again, the Lagrange multipliers rule tells us that the function $(x,y,z)\mapsto 2x^2+\frac{3}{2}y^2+\frac{5}{4}z^2$ attains its minimum on $\{(x,y,z)\in\mathbb{R}^3:x+y+z=n\}$ at a point $(x,y,z)$ with
$$
\left\{\begin{array}{ll}
4x & =\alpha \\ 3y & =\alpha \\ \frac{5}{2}z & = \alpha
\end{array}\right.
$$
for some $\alpha\in\mathbb{R}$.
Hence $n=\frac{\alpha}{4}+\frac{\alpha}{3}+\frac{2\alpha}{5}=\frac{59\alpha}{60}$. So $(x,y,z)=(\frac{15n}{59},\frac{20n}{59},\frac{24n}{59})$.
This yields
\begin{equation}\label{eq:ineq.for.n.Case:d4}
2(\lambda|\lambda)\geq 2\cdot\frac{(15n)^2}{59^2}+\frac{3}{2}\cdot\frac{(20n)^2}{59^2}+\frac{5}{4}\cdot\frac{(24n)^2}{59^2}-\frac{n^2}{2}=\frac{n^2}{118}>0,
\end{equation}
which shows that $(\lambda|\lambda)>0$ for all $\lambda$ in case (d4).

Moreover, from (\ref{eq:ineq.for.n.Case:d4}), we also obtain that $(\lambda|\lambda)>1$ for all $\lambda$ if $n\geq 16$. For $n\in\{12,13,14,15\}$, by exhausting all the cases, we see that the equality $(\lambda|\lambda)=1$ only happens for $n=12$ and $\lambda=((3,9),(2,2,8),(1^3,2,7),(6,6))$.

The proof of the theorem is complete.
\end{proof}

We emphasize the special case where the double flag variety $\Xfv=L/Q\times G/P$ is a product of four Grassmannian varieties.

\begin{corollary}\label{corollary:grassmannians}
Consider a double flag variety of the form
$$
\Xfv=L_1/Q_1 \times L_2/Q_2 \times L_3/Q_3 \times G/P
$$
where
$L=L_1\times L_2\times L_3$ is, as before, a Levi subgroup of $G=\GL_n$,
and assume that all the factors $L_i/Q_i$ are Grassmannian varieties
and $G/P=\Grass(d;n)$ is also the Grassmannian variety of $d$-dimensional subspaces in $\K^n$.
Then, $\Xfv$ has finitely many $L$-orbits in exactly the following two cases:
\begin{penumerate}
    \item $\min\{d,n-d\}\leq 2$;
    \item one of the Grassmannian varieties $L_i/Q_i \; (i = 1, 2, 3) $ is reduced to a point.
\end{penumerate}
\end{corollary}

\begin{proof}
By the assumption that all $ L_i/Q_i \; (i = 1, 2, 3) $ are Grassmannian varieties,
our double flag variety $\Xfv=\Xfv(\lambda)$ is associated with
a quadruple $\lambda=(\mbfa,\mbfb,\mbfc,\mbfd)$ of the form
$$
\lambda=((a_1,a_2,n-p),(b_1,b_2,n-q),(c_1,c_2,n-r),(d,n-d)).
$$
If we assume that $\Xfv$ has finitely many $L$-orbits, then
Theorem \ref{Thm:classification.finite.type.MFV}\,(3) tells us that either $\min\{d,n-d\}\leq 2$
(case (a) in Theorem \ref{Thm:classification.finite.type.MFV}\,(3))
or $\ell_i\in\{1,2\}$ for some $i\in\{1,2,3\}$,
which exactly means that
one of the factors $L_i/Q_i$ is reduced to a point.
Conversely, if $\min\{d,n-d\}\leq 2$ or $\ell_i\in\{1,2\}$ for some $i\in\{1,2,3\}$,
then one of the conditions (a), (o), (d3), (d4) in Theorem \ref{Thm:classification.finite.type.MFV}\,(3) is satisfied, which shows that $\Xfv$  has finitely many $L$-orbits.
\end{proof}

\begin{remark}
The case of a product of Grassmannian varieties, addressed in the corollary,
corresponds to the situation where
$\bfell=(3,3,3,2)$ in the notation of Section \ref{section:2}.
In this case, the semigroup $\Lambda(\bfell)$ contains a unique minimal element with nonpositive Tits form, namely
$$
\lambda^1=((1,1,4),(1,1,4),(1,1,4),(3,3))
$$
(i.e., $\lambda^2,\lambda^3$ involved in Theorem \ref{Thm:classification.finite.type.MFV} cannot be summands of any element in $\Lambda(\bfell)$).
Consequently, Theorem \ref{Thm:classification.finite.type.MFV} tells us that, for $\lambda\in\Lambda(\bfell)$,
$$
\mbox{$\Xfv(\lambda)$ is of finite type}
\quad\Longleftrightarrow\quad
\lambda\not\geq\lambda^1.
$$
It is easy to see that the relation $\lambda\not\geq\lambda^1$ holds exactly if one of the conditions (1) or (2) of Corollary \ref{corollary:grassmannians} is fulfilled, so that we again get the proof of the corollary in this way.
\end{remark}

\section{A parametrization of orbits}\label{section:4}

We follow The same setting as in Section \ref{section:classification.proofs}.
We say that a dimension vector $\lambda=(\mbfa,\mbfb,\mbfc,\mbfd)\in\Lambda(\bfell)$ (as in (\ref{form-lambda})) is {\it of finite type} if the double flag variety $\Xfv(\lambda)$ has finitely many $L$-orbits. Dimension vectors of finite type are therefore classified in Theorem \ref{Thm:classification.finite.type.MFV}.
Also, as it follows in general from Theorem \ref{theorem:finiteness.criterion}
(or as it appears in particular in Theorem \ref{Thm:classification.finite.type.MFV}), every summand $\mu$ of a dimension vector of finite type is of finite type.

If $\lambda$ is of finite type, then $(\lambda|\lambda)\geq 1$
(see Theorem \ref{theorem:finiteness.criterion}\,(1)).
Moreover, the dimension vectors $\lambda$ of finite type with Tits form $(\lambda|\lambda)=1$ play a key role in the parametrization of the orbits (see Theorem \ref{theorem:finiteness.criterion}\,(2));
the next theorem classifies these dimension vectors.

\begin{theorem}
\label{T:indecomposables}
Let $\bfell=(\ell_1,\ell_2,\ell_3,2)$ and let
$$
\lambda=(\mbfa,\mbfb,\mbfc,\mbfd)=((a_1,\ldots,a_{\ell_1}),(b_1,\ldots,b_{\ell_2}),(c_1,\ldots,c_{\ell_3}),(d_1,d_2))\in\Lambda(\bfell)
$$
be of finite type, with $n:=|\lambda|\geq 1$.
The following conditions are equivalent:
\begin{penumerate}
    \item $(\lambda|\lambda)=1$;
    \item $\lambda$ appears in one of Tables \ref{list.indecomposables:special}--\ref{list.indecomposables:odd} (up to permuting $\mbfa,\mbfb,\mbfc$, up to permutation within each tuple $\mbfa',\mbfb',\mbfc',\mbfd$, and up to removing the coefficients of $\mbfa',\mbfb',\mbfc'$ which are equal to $0$).
\end{penumerate}

\begin{table}[H]
    $\begin{array}{ccccc|c}
    \lambda: & \mbfa & \mbfb & \mbfc & \mbfd & n=|\lambda|
    \\[1mm] \hline
    & (1) & (1) & (1,0) & (0, 1) & 1
    \\[1mm] \hline
    & (2) & (1,1) & (1,1) & (1,1) & 2
    \\[1mm] \hline
    & (1,2) & (1,2) & (1,2) & (1,2) & 3
    \\[1mm] \hline
    & (1,3) & (1,3) & (1,1,2) & (2,2) & 4
    \\[1mm] \hline
    & (1,4) & (1,1,3) & (1,1,3) & (2,3) & 5
    \\[1mm] \hline
    & (1,1,4) & (1,1,4) & (1,1,4) & (2,4) & 6
    \\[1mm] \cline{1-5}
    & (1,5) & (1,1,4) & ( 1^3,3) & (3,3) &
    \\[1mm] \cline{1-5}
    & (2,4) & (1,1,4) & (1,1,4) & (3,3) &
    \\[1mm] \hline
    & (1,6) & (1^3,4) & (1^3,4) & (3,4) & 7
    \\[1mm] \cline{1-5}
    & (2,5) & (1,1,5) & (1^3,4) & (3,4) &
    \\[1mm] \hline
    & (2,6) & (1^3,5) & (1^3,5) & (3,5) & 8
    \\[1mm] \cline{1-5}
    & (2,6) & (1,1,6) & (1^4,4) & (4,4) &
    \\[1mm] \cline{1-5}
    & (2,6) & (1,2,5) & (1^3,5) & (4,4)
    \\[1mm] \hline
    & (2,7) & (1,2,6) & (1^4,5) & (4,5) & 9
    \\[1mm] \hline
    & (2,8) & (1,2,7) & (1^5,5) & (5,5)
    \\[1mm] \cline{1-5}
    & (2,8) & (2,2,6) & (1^4,6) & (5,5) & 10
    \\[1mm] \cline{1-5}
    & (3,7) & (1,2,7) & (1^4,6) & (5,5) &
    \\[1mm] \hline
    & (2,9) & (2,2,7) & (1^5,6) & (5,6) & 11
    \\[1mm] \cline{1-5}
    & (3,8) & (1,2,8) & (1^5,6) & (5,6) &
    \\[1mm] \cline{1-5}
    & (3,8) & (2,2,7) & (1^4,7) & (5,6) &
    \\[1mm] \hline
    & (3,9) & (2,2,8) & (1^5,7) & (5,7) & 12
    \\[1mm] \hline
    & (2,10) & (2,2,8) & (1^6,6) & (6,6) &
    \\[1mm] \hline
    & (3,9) & (2,2,8) & (1^3,2,7) & (6,6) &
    \\[1mm] \hline
    \end{array}$
\caption{Dimension vectors $\lambda$ of finite type with $(\lambda|\lambda)=1$: special cases ($|\lambda|\leq 12$)}
\label{list.indecomposables:special}
\end{table}

\begin{table}[H]
{\footnotesize
$\begin{array}{ccccc|c}
    \lambda: & \mbfa & \mbfb & \mbfc & \mbfd & n=|\lambda|
    \\[1mm] \hline
    & (1,2 k - 1) & (1^{k - 1},k + 1) & (1^k,k) & (k,k) & 2k\ (k\geq 4)
    \\[1mm] \hline
    & (m, 3 m) & (1, m - 1, 3 m) & (1^{2m}, 2m) & (2m, 2m) & 4m\ (m\geq 3)
    \\[1mm] \cline{1-5}
    & (m, 3 m) & (1, m, 3 m - 1) & (1^{2m - 1}, 2m + 1) & (2m, 2m)
    \\[1mm] \hline
    & (m, 3 m + 2) & (1, m, 3 m + 1) & (1^{2m + 1}, 2m + 1) & (2m + 1, 2m + 1) & 4m+2\ (m\geq 3)
    \\[1mm] \cline{1-5}
    & (m + 1, 3 m + 1) & (1, m, 3 m + 1) & (1^{2m}, 2m + 2) & (2m + 1, 2m + 1)
    \\[1mm] \hline
    \end{array}$}
\caption{Dimension vectors $\lambda$ of finite type with $(\lambda|\lambda)=1$: infinite series, $|\lambda|$ even}
\label{list.indecomposables:even}
\end{table}

\begin{table}[H]
{\footnotesize
$\begin{array}{ccccc|c}
    \lambda: & \mbfa & \mbfb & \mbfc & \mbfd & n=|\lambda|
    \\[1mm] \hline
    & (1,2 k) & (1^k,k + 1) & (1^k,k + 1) & (k,k + 1) & 2k+1\ (k\geq 4)
    \\[1mm] \hline
    & (m, 3 m + 1) & (1, m, 3 m) & (1^{2m}, 2m + 1) & (2 m , 2m + 1) & 4m+1\ (m\geq 3)
    \\[1mm] \hline
    & (m + 1, 3 m + 2) & (1, m, 3 m + 2) & (1^{2m + 1}, 2m + 2) & (2 m + 1, 2m + 2) & 4m+3\ (m\geq 3)
    \\[1mm] \hline
    \end{array}$}
\caption{Dimension vectors $\lambda$ of finite type with $(\lambda|\lambda)=1$: infinite series, $|\lambda|$ odd}
\label{list.indecomposables:odd}
\end{table}
\end{theorem}

\begin{proof}
The classification follows from the analysis of the
cases of equality $(\lambda|\lambda)=1$ which are
made in the course of the proof of the parts (o), (a)--(c), (d1)--(d4) for the implication (3)$\Rightarrow$(1) of Theorem \ref{Thm:classification.finite.type.MFV}.
\end{proof}

\begin{corollary}
\label{C:orbits}
Let $\bfell=(\ell_1,\ell_2,\ell_3,2)$ and let
$\lambda=(\mbfa,\mbfb,\mbfc,\mbfd)\in\Lambda(\bfell)$
be a dimension vector of finite type.
Then, there is a one-to-one correspondence between the $L$-orbits of $\Xfv(\lambda)$ and the decompositions (up to permutation of the terms)
$\lambda=\sum_{i=1}^k \mu_i$ into summands $\mu_i\in\Lambda(\bfell)$ satisfying the equivalent conditions (1)--(2) of Theorem \ref{T:indecomposables}.
\end{corollary}

\begin{proof}
This follows from Theorem \ref{theorem:finiteness.criterion}\,(2).
\end{proof}

\section{Example: orbits in the case of the Grassmannian variety $ G/P = \Grass(2;n) $}

\label{section:5}

In this section we consider the case where $\mbfd=(2,n-2)$.
By Theorem \ref{Thm:classification.finite.type.MFV}, we know that:

\begin{proposition}
The double flag variety $\Xfv(\mbfa,\mbfb,\mbfc,(2,n-2))$
is of finite type, for any choice of $\mbfa,\mbfb,\mbfc$.
In other words, $\Grass(2;n)$ is an $L$-spherical variety
whenever $L=\GL_p\times \GL_q\times \GL_r$ with $p+q+r=n$.
\end{proposition}

In order to give a better insight of this proposition, we aim to describe explicitly the orbits of $\Xfv=\Xfv(\mbfa,\mbfb,\mbfc,(2,n-2))=L/Q\times \Grass(2;n)$
in the extreme case where $Q=B_L$ is a Borel subgroup of $L$. Hence
$$
\Xfv=\Flags(1^p)\times \Flags(1^q)\times \Flags(1^r)\times \Grass(2;n)
$$
where $\Flags(1^p),\Flags(1^q),\Flags(1^r)$ are full flag varieties.
The $L$-orbits of $\Xfv$ naturally correspond to the $B_L$-orbits of $\Grass(2;n)$.

Specifically, we fix a standard decomposition
$V:=\K^n=A\oplus B\oplus C$
where $A,B,C$ have respective dimensions $p,q,r$. Then
$\Flags(1^p),\Flags(1^q),\Flags(1^r)$ are understood as the varieties of complete flags of $A,B,C$, respectively.
Let $(e^A_1,\ldots,e^A_p)$ be a fixed basis of $A$, and we fix similarly bases of $B$ and $C$.
Let $F^A=(\langle e^A_1,\ldots,e^A_i\rangle)_{i=1}^p$
be the standard flag of $A$, and define similarly $F^B$ and $F^C$, so that
$B_L$ becomes the stabilizer of the triple $(F^A,F^B,F^C)$.

Every $L$-orbit of $\Xfv$ has a representative of the form $(F^A,F^B,F^C,M)$, where $M\in\Grass(2;n)$
is a representative of the corresponding $B_L$-orbit of $\Grass(2;n)$. Our purpose is to parametrize these orbits and to specify a subspace $M$ attached to each orbit.

In order to parametrize the orbits, by Corollary \ref{C:orbits}, we need to determine the decompositions (in the semigroup $\Lambda(\bfell)$ for $\bfell=(p+1,q+1,r+1,2)$) of the dimension vector
\begin{equation} \label{dimension.vector:lambda}
\lambda=((1^p,n-p),(1^q,n-q),(1^r,n-r),(2,n-2))
\end{equation}
into summands $\mu$ satisfying the conditions of Theorem \ref{T:indecomposables}.
We will speak of a decomposition of $\lambda$ into {\em indecomposable summands}.

The possible summands $\mu=(\mbfa,\mbfb,\mbfc,\mbfd)$ that can arise in the decomposition of $\lambda$ are obtained from Theorem \ref{T:indecomposables} and we list them (up to adding $0$'s within $\mbfa',\mbfb',\mbfc'$ and permuting $\mbfa,\mbfb,\mbfc$) in Table \ref{list.indecomposables:Gr(2;n)}.
Also for every such $\mu$ we indicate the number $|\mu|=|\mbfa|=|\mbfb|=|\mbfc|=|\mbfd|$ as well as the composition $(p_\mu,q_\mu,r_\mu):=(|\mbfa'|,|\mbfb'|,|\mbfc'|)$.
\begin{table}[H]
    $\begin{array}{c|cccc|c|c|c}
     & \mbfa & \mbfb & \mbfc & \mbfd & (p_\mu,q_\mu,r_\mu) & |\mu| & M_\mu
    \\[1mm] \hline
    \mu_1 & (1) & (1) & (1,0) & (0,1)  & (0,0,1) & 1 & \{0\}
    \\[1mm] \hline
    \mu'_1 & (1) & (1) & (1,0) & (1,0)  & (0,0,1) & 1 & \langle \varepsilon_1^C\rangle
    \\[1mm] \hline
    \mu_2 & (2) & (1,1) & (1,1) & (1,1)   & (0,1,1) & 2 & \langle \varepsilon_1^B+\varepsilon_1^C\rangle
    \\[1mm] \hline
    \mu_3 & (1,2) & (1,2) & (1,2) & (1,2)   & (1,1,1) & 3 & \langle \varepsilon_1^A+\varepsilon_1^B+\varepsilon_1^C\rangle
    \\[1mm] \hline
    \mu'_3 & (1,2) & (1,2) & (1,2) & (2,1)   & (1,1,1) & 3
    & \langle \varepsilon_1^A+\varepsilon_1^B,\varepsilon_1^B+\varepsilon_1^C\rangle
    \\[1mm] \hline
    \mu_4 & (1,3) & (1,3) & (1,1,2) & (2,2)   & (1,1,2) & 4 & \langle \varepsilon_1^A+\varepsilon_1^B+\varepsilon_1^C,
    \varepsilon_1^A-\varepsilon_1^B+\varepsilon_2^C \rangle
    \\[1mm] \hline
    \mu_5 & (1,4) & (1,1,3) & (1,1,3) & (2,3)   & (1,2,2) & 5
    & \langle \varepsilon_1^A+\varepsilon_1^B+\varepsilon_2^C,
    \varepsilon_1^A+\varepsilon_2^B+\varepsilon_1^C \rangle
    \\[1mm] \hline
    \mu_6 & (1,1,4) & (1,1,4) & (1,1,4) & (2,4)   & (2,2,2) & 6
    & \langle \varepsilon_1^A+\varepsilon_2^B+\varepsilon_2^C,
    \varepsilon_2^A+\varepsilon_1^B+\varepsilon_1^C+\varepsilon_2^C \rangle
    \\[1mm] \hline
    \end{array}$
\caption{Possible indecomposable summands $\mu$ of $\lambda$}
\label{list.indecomposables:Gr(2;n)}
\end{table}

For every $\mu$ in Table \ref{list.indecomposables:Gr(2;n)}, the variety $\Xfv(\mu)$ is of finite type and has a unique dense orbit for the action of the appropriate Levi subgroup
$$L(\mu):=\GL_{p_\mu}\times \GL_{q_\mu}\times \GL_{r_\mu}\cong \GL(A_\mu)\times\GL(B_\mu)\times\GL(C_\mu)\subset \GL(V_\mu)\cong\GL_{|\mu|},$$
where $V_\mu:=\K^{|\mu|}$. 
In fact $\Xfv(\mu)$ is of the form
$\Flags(A_\mu)\times\Flags(B_\mu)\times\Flags(C_\mu)\times\Grass(d_1;V_\mu)$
and its dense orbit naturally corresponds to a dense $B(\mu)$-orbit of $\Grass(d_1;V_\mu)$, where $B(\mu)\subset L(\mu)$ is the standard (upper-triangular) Borel subgroup. In the table, for every $\mu$, we have also indicated a representative $M_\mu$ of that dense $B(\mu)$-orbit. To do this, we denote by $(\varepsilon_i^A)_{i=1}^{p_\mu}, (\varepsilon_j^B)_{j=1}^{q_\mu},(\varepsilon_k^C)_{k=1}^{r_\mu}$ the standard bases of $A_\mu,B_\mu,C_\mu$.

An indecomposable summand $\mu=(\mbfa,\mbfb,\mbfc,(d_1,d_2))\in\Lambda(\bfell)$ of $\lambda$
satisfies $d_1\in\{0,1,2\}$. We say that the summand $\mu$ is of type $0$, $1$, or $2$, depending on the value of $d_1$.
Equivalently, after removing the $0$'s from $\mbfa',\mbfb',\mbfc'$ and up to permuting $\mbfa,\mbfb,\mbfc$,
$\mu$ is of type $0$ if and only if $ \mu = \mu_1 $ in Table \ref{list.indecomposables:Gr(2;n)};
$\mu$ is of type $1$ if and only if $ \mu = \mu_1',\mu_2$ or $\mu_3$ in Table \ref{list.indecomposables:Gr(2;n)};
$\mu$ is of type $2$ if and only if $ \mu = \mu_3',\mu_4,\mu_5$ or $\mu_6$ in Table \ref{list.indecomposables:Gr(2;n)}.

Equivalently, $\mu$ is of type $0$, $1$, or $2$, if and only if, after removing the $0$'s from $\mbfa',\mbfb',\mbfc'$ and up to permuting $\mbfa,\mbfb,\mbfc$, $\mu$ coincides with $\mu_1$, respectively, $\mu_1',\mu_2$ or $\mu_3$, respectively, $\mu_3',\mu_4,\mu_5$ or $\mu_6$, with the notation of Table \ref{list.indecomposables:Gr(2;n)}.
A decomposition of $\lambda$ into indecomposable summands
\begin{itemize}
    \item either contains exactly two summands of type $1$ and no summand of type $2$,
    \item or contains no summand of type $1$ and exactly one summand of type $2$.
\end{itemize}
The rest of the summands are of type $0$ (and they are completely determined by the previous summands of type $1$ or $2$). In the former case we say that the decomposition of $\lambda$ is of type $1$, in the latter case
we say that it is of type $2$.

Decompositions of $\lambda$ can be encoded by the following combinatorial data. We call {\em colored diagram of shape $(p,q,r)$}
a diagram $\eta$ with three rows of, respectively, $p,q,r$ empty boxes, some of them being colored according to the following rule:
\begin{itemize}
    \item either $\eta$ contains two disjoint sets of gray-colored boxes, each set containing at most one box in each row (we use two interchangeable shades of gray, one for each set -- the colored diagram is considered the same if we switch the two shades of gray);
    \item or $\eta$ contains a single set of black-colored boxes, this set containing at least one and at most two boxes in each row.
\end{itemize}
Every such colored diagram $\eta$ gives rise to a decomposition of $\lambda$.
\begin{itemize}
\item Each set of colored boxes in $\eta$ corresponds to an indecomposable summand $\mu=(\mbfa,\mbfb,\mbfc,\mbfd)$ of $\lambda$ of type $1$ (in case of a set of gray boxes) or of type $2$ (in case of black boxes).
\begin{itemize}
\item The number of colored boxes in the set is the number $|\mu|$ attached to the summand. 
\item The positions of the colored boxes indicate the  nonzero coefficients, equal to $1$, in $\mbfa',\mbfb',\mbfc'$; this determines $\mbfa=(\mbfa',|\mu|-|\mbfa'|)$ and similarly $\mbfb$ and $\mbfc$.
\item We have $\mbfd=(1,|\mu|-1)$ in case of a set of gray boxes, or $\mbfd=(2,|\mu|-2)$ in case of black boxes.
\end{itemize}
\item The blank boxes of $\eta$ correspond to summands of type $0$, which complete the decomposition.
\begin{itemize}
    \item Every summand $\mu$ of type $0$ is such that $|\mu|=1$, $\mbfd=(0,1)$, and there is a single nonzero coefficient in $\mbfa'$, $\mbfb'$, or $\mbfc'$, equal to $1$, determined by the position of the considered blank box of $\eta$.
\end{itemize}
\end{itemize}
Conversely, every decomposition of $\lambda$ into indecomposable summands is obtained in this way, from some colored diagram of shape $(p,q,r)$.

\begin{example}
\label{Examples:colored.diagrams}
(a) If $\lambda=((1,1,3),(1,1,3),(1,4),(2,3))$, the colored diagram
$$
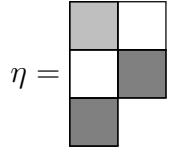
\eta=\ydiagram{2,2,1}
     *[*(white)]{0,1,0}
     *[*(lightgray)]{1,0,0} *[*(gray)]{0,2,1}
$$
has two sets of gray boxes which give rise to two indecomposable summands of type $1$, namely
$\mu=((1,0,0),(0,0,1),(0,1),(1,0))$
and $\mu'=((0,0,2),(0,1,1),(1,1),(1,1))$,
while the two blank boxes give rise to the summands
$\nu=((0,1,0),(0,0,1),(0,1),(0,1))$ and $\nu'=((0,0,1),(1,0,0),(0,1),(0,1))$ of type $0$.
Hence the colored diagram corresponds to the decomposition
\begin{eqnarray*}
\lambda & = & \mu+\mu'+\nu+\nu' \\
 & = & ((1,0,0),(0,0,1),(0,1),(1,0))+((0,0,2),(0,1,1),(1,1),(1,1))+(\ldots).
\end{eqnarray*}
On the second line, we only indicate the summands of type $1$, since the remaining summands are of type $0$ (incorporated in the symbol ``$(\ldots)$'') are in fact completely determined by them.
We use the same convention in the next examples.

(b) Let $\lambda=((1^3,6),(1^4,5),(1^2,7),(2,7))$, hence $(p,q,r)=(3,4,2)$ and $n=9$.
The colored diagrams
$$
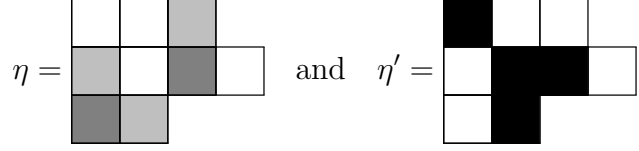
\eta=\ydiagram{3,4,2}
     *[*(lightgray)]{0,1,0}
     *[*(white)]{2,2,0}
     *[*(gray)]{0,0,1}
     *[*(lightgray)]{3,0,2} *[*(gray)]{0,3,1}\quad\text{and}\quad \eta'=\ydiagram{3,4,2}
     *[*(white)]{0,1,1} *[*(black)]{1,3,2}
$$
respectively correspond to the decompositions
$$
\lambda=((0,0,1,2),(1,0,0,0,2),(0,1,2),(1,2))+((0^3,2),(0,0,1,0,1),(1,0,1),(1,1))+(\ldots)
$$
(which is a decomposition of type $1$)
and
$$
\lambda=((1,0,0,3),(0,1,1,0,2),(0,1,3),(2,2))+(\ldots)
$$
(which is a decomposition of type $2$).
\end{example}

We now attach an element $M_\eta\in\Grass(2;n)$ to every colored diagram $\eta$ of shape $(p,q,r)$.
Every set $S$ of colored boxes in $\eta$ corresponds to an indecomposable summand of $\lambda$, and we rely on the subspaces $M_\mu$ associated to the indecomposable summands and described in Table \ref{list.indecomposables:Gr(2;n)}.

The basis $(e_i^A)_{i=1}^p\cup (e_j^B)_{j=1}^q\cup (e_k^C)_{k=1}^r$ of the space $V=\K^n$ can be indexed on the boxes of $\eta$ so that $e_1^A,\ldots,e_p^A$ correspond to the boxes of the first row, and similarly $(e_j^B)_{j=1}^q$ and $(e_k^C)_{k=1}^r$ correspond to the second and to the third row;
we denote by $e_x$ the basis vector corresponding to the box $x$ of $\eta$.

If $\eta$ contains two sets $S_1$ and $S_2$ of gray-colored boxes, we define
$$
M_\eta=\langle \sum_{x\in S_1}e_x,\sum_{x\in S_2}e_x\rangle.
$$

If $\eta$ contains a single set $S$ of black-colored boxes, then $S$ contains at least one element per row. Let $x,y,z$ be the left-most elements of $S$ appearing in each row, named so that the row of $x$ has a minimal number of elements of $S$ and the row of $z$ has a maximal number of elements of $S$ -- if two rows have the same number of elements of $S$, we name first the element in the upper row. When it exists, let $x'$ be the second element of $S$ in the same row as $x$, and define similarly $y'$ and $z'$.
\begin{itemize}
    \item If $\#S=3$, i.e., $S=\{x,y,z\}$, then set
    $M_\eta=\langle e_{x}+e_{y},e_{y}+e_{z}\rangle$.
    \item If $\#S=4$, i.e., $S=\{x,y,z,z'\}$, then set
    $M_\eta=\langle e_{x}+e_{y}+e_{z},e_{x}-e_{y}+e_{z'}\rangle$.
    \item If $\#S=5$, i.e., $S=\{x,y,y',z,z'\}$, then set
    $M_\eta=\langle e_{x}+e_{y}+e_{z'},e_{x}+e_{y'}+e_{z}\rangle$.
    \item If $\#S=6$, i.e., $S=\{x,x',y,y',z,z'\}$, set
    $M_\eta=\langle e_{x}+e_{y'}+e_{z'},e_{x'}+e_{y}+e_{z}+e_{z'}\rangle$.
\end{itemize}

\begin{example}
For $\eta$ as in Example \ref{Examples:colored.diagrams}\,(a), we have
$$
M_\eta=\langle e_1^A,e_2^B+e_1^C \rangle.
$$
For $\eta$ and $\eta'$ as in Example \ref{Examples:colored.diagrams}\,(b), we obtain
$$
M_\eta=\langle e_3^A+e_1^B+e_2^C,e_3^B+e_1^C\rangle,\quad
M_{\eta'}=\langle e_1^A+e_2^C+e_2^B,e_1^A-e_2^C+e_3^B\rangle.
$$
\end{example}

The next result summarizes the constructions made in this section.
It relies on Corollary \ref{C:orbits}, taking also Proposition \ref{proposition:orbits.in.xfv.and.RepQ'} into account.

\begin{theorem}
Let $n=p+q+r\geq 2$, $L=\GL_p\times\GL_q\times\GL_r$, and let $B_L\subset L$ be the Borel subgroup of upper-triangular matrices.
The $L$-orbits of $\Flags(1^p)\times\Flags(1^q)\times\Flags(1^r)\times\Grass(2;n)$, and thus also the $B_L$-orbits of $\Grass(2;n)$,
are parametrized by the colored diagrams of shape $(p,q,r)$.
Specifically, the map
$$
\{\text{\rm colored diagrams of shape $(p,q,r)$}\}\to \Grass(2;n)/B_L,\quad \eta\mapsto B_L\cdot M_\eta
$$
is a bijection.
\end{theorem}

\printbibliography

@article {Avdeev.Petukhov.2014,
    AUTHOR = {Avdeev, R. S. and Petukhov, A. V.},
     TITLE = {Spherical actions on flag varieties},
   JOURNAL = {Mat. Sb.},
  FJOURNAL = {Matematicheski\u{\i} Sbornik},
    VOLUME = {205},
      YEAR = {2014},
    NUMBER = {9},
     PAGES = {3--48},
      ISSN = {0368-8666},
   MRCLASS = {14L30 (14M15 14M27)},
  MRNUMBER = {3288423},
MRREVIEWER = {Dmitry A. Timash\"{e}v},
       DOI = {10.1070/sm2014v205n09abeh004416},
}

@article {MWZ.2000,
    AUTHOR = {Magyar, Peter and Weyman, Jerzy and Zelevinsky, Andrei},
     TITLE = {Symplectic multiple flag varieties of finite type},
   JOURNAL = {J. Algebra},
  FJOURNAL = {Journal of Algebra},
    VOLUME = {230},
      YEAR = {2000},
    NUMBER = {1},
     PAGES = {245--265},
      ISSN = {0021-8693},
     CODEN = {JALGA4},
   MRCLASS = {14M15 (14L30 16G20)},
  MRNUMBER = {MR1774766 (2001i:14064)},
MRREVIEWER = {Dmitri I. Panyushev},
}

@article {MWZ.1999,
    AUTHOR = {Magyar, Peter and Weyman, Jerzy and Zelevinsky, Andrei},
     TITLE = {Multiple flag varieties of finite type},
   JOURNAL = {Adv. Math.},
  FJOURNAL = {Advances in Mathematics},
    VOLUME = {141},
      YEAR = {1999},
    NUMBER = {1},
     PAGES = {97--118},
      ISSN = {0001-8708},
     CODEN = {ADMTA4},
   MRCLASS = {14M15 (16G20)},
  MRNUMBER = {MR1667147 (99m:14095)},
MRREVIEWER = {Michel Brion},
}

@preamble{
   "\def\cprime{$'$} "
}

@article {Matsuki.2015,
    AUTHOR = {Matsuki, Toshihiko},
     TITLE = {Orthogonal multiple flag varieties of finite type {I}: {O}dd
              degree case},
   JOURNAL = {J. Algebra},
  FJOURNAL = {Journal of Algebra},
    VOLUME = {425},
      YEAR = {2015},
     PAGES = {450--523},
      ISSN = {0021-8693},
   MRCLASS = {14M15 (14L30 20G15)},
  MRNUMBER = {3295993},
MRREVIEWER = {Justin Brown},
       DOI = {10.1016/j.jalgebra.2014.11.016},
}

@article{Matsuki.arXiv2019,
  author = {Matsuki, Toshihiko},
  title = {{Orthogonal multiple flag varieties of finite type II : even degree case}},
  publisher = {arXiv},
  year = {2019},
  copyright = {arXiv.org perpetual, non-exclusive license},
  doi = {10.48550/ARXIV.1903.06335},
  journal={{\upshape\texttt{arXiv:1903.06335}}},
}

@preamble{
   "\def\Dbar{\leavevmode\lower.6ex\hbox to 0pt{\hskip-.23ex
    \accent"16\hss}D} "
}

@article {Kac.Quiver.1980,
    AUTHOR = {Kac, V. G.},
     TITLE = {Infinite root systems, representations of graphs and invariant
              theory},
   JOURNAL = {Invent. Math.},
  FJOURNAL = {Inventiones Mathematicae},
    VOLUME = {56},
      YEAR = {1980},
    NUMBER = {1},
     PAGES = {57--92},
      ISSN = {0020-9910},
   MRCLASS = {16A64 (05C99 14D25 17B20 17B65 17B70)},
  MRNUMBER = {557581},
MRREVIEWER = {C. M. Ringel},
       DOI = {10.1007/BF01403155},
}

@article{Homma.2021,
  title={Double Flag Varieties and Representations of Quivers},
  author={Homma, Hiroki},
  journal={{\upshape\texttt{arXiv: 2103.\-14509}}},
  year={2021}
}

@misc{Fresse.Nishiyama.Overview.2023,
      title={Overview on the theory of double flag varieties for symmetric pairs}, 
      author={Lucas Fresse and Kyo Nishiyama},
      year={2023},
      eprint={2309.17085},
      archivePrefix={arXiv},
      primaryClass={math.RT}
}

@article {Duckworth,
    AUTHOR = {Duckworth, W. Ethan},
     TITLE = {A classification of certain finite double coset collections in
              the classical groups},
   JOURNAL = {Bull. London Math. Soc.},
  FJOURNAL = {The Bulletin of the London Mathematical Society},
    VOLUME = {36},
      YEAR = {2004},
    NUMBER = {6},
     PAGES = {758--768},
      ISSN = {0024-6093,1469-2120},
   MRCLASS = {20G15 (14L35)},
  MRNUMBER = {2083751},
MRREVIEWER = {Dan\ Edidin},
       DOI = {10.1112/S0024609304003546},
       URL = {https://doi.org/10.1112/S0024609304003546},
}

\end{document}